\documentclass[a4paper,12pt,reqno,oneside]{amsart}

\usepackage{setspace} 
\usepackage{typearea} 

\usepackage{amscd,amsmath,amssymb,verbatim}
\usepackage{xr-hyper}
\usepackage{graphicx}
\usepackage{ifthen}
\usepackage{cmll}
\usepackage{paralist,cite}
\usepackage[colorlinks,backref,final,bookmarksnumbered,bookmarks]{hyperref}
\usepackage[backrefs]{amsrefs}
\usepackage{tikz-cd}

\usepackage[utf8]{inputenc}
\usepackage[T1,T2A]{fontenc}

\newcommand{\arxiv}[2][]{\ifthenelse{\equal{#1}{}}
{\href{http://arxiv.org/abs/#2}{\tt arXiv:#2}}
{\href{http://arxiv.org/abs/math/#2}{\tt arXiv:math.#1/#2}}}

\theoremstyle{plain}
\newtheorem{maintheorem}{Theorem}
 
\newtheorem{theorem}{Theorem}[section]
\newtheorem{lemma}[theorem]{Lemma}

\theoremstyle{definition}
\newtheorem{example}[theorem]{Example}

\newtheoremstyle{remark}
{}{}{}{}{\itshape}{}{ }{\thmname{#1}\thmnumber{ \itshape #2.}}
\theoremstyle{remark}
\newtheorem{remark}[theorem]{Remark}

\newtheoremstyle{concise}
{}{}{}{}{\bfseries}{}{ }{\thmnumber{#2.}\thmnote{ #3.}}
\theoremstyle{concise}

\newenvironment{embedded roster}[1][0]
{

\begin{enumerate}\setcounter{enumii}{#1}}
{\end{enumerate}}
\makeatletter
\renewcommand\p@enumii{\p@enumi}
\renewcommand{\@listii}{\leftmargin=40pt}
\makeatother

\def\N{\mathbb{N}} 
\def\R{\mathbb{R}} 
\def\Z{\mathbb{Z}} 

\def\F{\mathcal{F}} 
 
\def\h{\mathcal{H}}

\def\x{\times}
\def\but{\setminus}

\def\C{\mathcal{C}}
\def\eps{\varepsilon} 
\def\phi{\varphi} 
\def\invlim{\lim\nolimits} 
\def\xr#1{\xrightarrow{#1}} 
 \renewcommand{\:}{\colon}

\DeclareMathOperator*{\colim}{colim}

 \DeclareMathOperator{\id}{id}
  
 \DeclareMathOperator{\Hom}{Hom} 
\DeclareMathOperator{\hocolim}{hocolim} \DeclareMathOperator{\holim}{holim}

\DeclareMathOperator{\ost}{ost} 
\DeclareMathOperator{\cyl}{cyl} \DeclareMathOperator{\cocyl}{cocyl}
\DeclareMathOperator{\Lim}{Lim} 

\DeclareMathAlphabet{\eu}{OT1}{cmss}{m}{sl}
\def\Hocolim{\eu h\eu o\eu c\eu o\eu l\eu i\eu m\,} 
\def\Holim{\eu h\eu o\eu l\eu i\eu m\,}

\def\tph#1{\raise2.5pt\hbox{\the\textfont1\char"7F}\!\!#1}
\def\tpm#1{\raise0pt\hbox{\the\textfont1\char"7F}\!#1}
\def\tpl#1{\lower1.5pt\hbox{\the\textfont1\char"7F}\!#1}

\def\Cel{\lfloor} \def\Cer{\rfloor}
 
\def\downscale#1{\mathchoice{\raisebox{1pt}{$\scriptstyle#1$}}
{\raisebox{1pt}{$\scriptstyle#1$}}{\raisebox{.5pt}{$\scriptscriptstyle#1$}}
{{\scriptscriptstyle#1}}}
\def\cel{\downscale\Cel} 
\def\cer{\downscale\Cer}

\def\bydef{\mathrel{\mathop:}=}

\def\nee#1{\raisebox{-1.5pt}{$\not\overset{\scriptstyle#1}=$}}

\begin{document}

\title[Lim colim versus colim lim. II: Derived limits over a pospace]
{Lim colim versus colim lim. \\ II: Derived limits over a pospace}
\author{Sergey A. Melikhov}
\address{Steklov Mathematical Institute of Russian Academy of Sciences,
ul.\ Gubkina 8, Moscow, 119991 Russia}
\email{melikhov@mi-ras.ru}

\begin{abstract}
\v Cech cohomology $H^n(X)$ of a separable metrizable space $X$ is defined in terms of cohomology 
of its nerves (or ANR neighborhoods) $P_\beta$, whereas Steenrod--Sitnikov homology $H_n(X)$ 
is defined in terms of homology of compact subsets $K_\alpha\subset X$.

We show that one can also go vice versa: in a sense, $H^n(X)$ can be reconstructed from $H^*(K_\alpha)$, and 
if $X$ is finite dimensional, $H_n(X)$ can be reconstructed from $H_*(P_\beta)$.

The reconstruction is via a Bousfield--Kan/Araki--Yoshimura type spectral sequence of the form
$\lim^p H^q(K_\alpha)\Rightarrow H^{p+q}(X)$, respectively $\lim^p H_q(P_\beta)\Rightarrow H_{q-p}(X)$, 
except that the derived limits have to be ``corrected'' so as to take into account a natural topology 
on the indexing set.
The corrected derived limits coincide with the usual ones when the topology is discrete, and in general
are applied not to an inverse system but to a ``partially ordered sheaf''.

The ``correction'' of the derived limit functors in turn involves constructing a ``correct'' (metrizable) 
topology on the order complex $|P|$ of a partially ordered metrizable space $P$ (such as the hyperspace
$K(X)$ of nonempty compact subsets of $X$ with the Hausdorff metric). 
It turns out that three natural approaches (by using the space of measurable functions, the space of probability 
measures, or the usual embedding $K(X)\to C(X;\R)$) all lead to the same topology on $|P|$.
\end{abstract}

\maketitle
\section{Introduction}

Let $X$ be a separable metrizable space, let $K_\alpha$ run over its compact subsets, and let $P_\beta$ run 
over the nerves of open covers of $X$ (or over open neighborhoods of $X$ in some ANR).
The Steenrod--Sitnikov homology and the \v Cech cohomology of $X$ are defined by
\[H_n(X)\bydef \colim_\alpha H_n(K_\alpha),\]
\[H^n(X)\bydef \colim_\beta H^n(P_\beta),\]
where $H_n(K_\alpha)$ is Steenrod homology (concerning the latter see e.g.\ \cite{M1} or \cite{M00}).

It is very natural to try to do everything vice versa, and define some kind of homology of $X$ in terms of
$H_n(K_\alpha)$ and some kind of cohomology of $X$ in terms of $H^n(P_\beta)$.
Of course, instead of the direct limits ($\colim$) it would be natural to use inverse limits ($\lim$), as well as 
their standard ``correction terms'' --- higher derived limits $\lim^p$, $p>0$.

\subsection{Strong (co)homology} 
This natural path leads rather unambiguously to ``strong homology'' and ``strong cohomology'', which were
introduced by Lisica and Marde\v si\'c and independently by Miminoshvili in the 1980s.
Despite doubtless naturality of their construction, there is a big trouble with these ``strong'' groups: they cannot 
be computed in ZFC for the very simplest examples.

\begin{example} \label{ZFC}
Let $\N$ denote the countable discrete space and $\N^+$ its one-point compactification.
The space $\N\x\N^+$ is arguably the very simplest non-compact, non-triangulable space.
(Algebraic topology of simplicial complexes and algebraic topology of compacta are completely different 
stories, which are much, much better understood.)
But the $(-1)$-dimensional strong homology of $\N\x\N^+$ cannot be computed in ZFC \cite{MP}.
Of course, another issue is that one should not really be computing $(-1)$-dimensional homology at all;
but if $\N^+$ is replaced with the $n$-dimensional Hawaiian earring $(\R^n\x\N)^+$, we have exactly 
the same problem with the $(n-1)$-dimensional strong homology \cite{MP}.
Also, the $(n+1)$-dimensional strong cohomology of the metric quotient $\N\x(\R^n\x\N)^+/(\N\x\infty)$,
which can alternatively be described as the non-compact cluster (=metric wedge) $\bigvee_{i\in\N} (\R^n\x\N)$,
cannot be computed in ZFC for the same reason (see \cite{M-III}*{Example \ref{lim:MP'}}).
\end{example}

This strange issue makes strong homology and cohomology much more interesting for the purposes of Foundations 
of Mathematics, but at the same time obviously ``defective'' for the purposes of geometric topology: at least 
with the current level of depth of human thought, it appears that statements carrying genuine geometric meaning 
have too low logical complexity to have a chance of being independent of ZFC.
Nevertheless, over the years there has appeared a considerable amount of literature about strong homology and 
strong cohomology (see \cite{MP}, \cite{Mard}, \cite{Li2} and references there).
One attractive feature of strong homology is its built-in invariance under strong shape.
It is an open problem whether Steenrod--Sitnikov homology is an invariant of strong shape.
However, as fine shape is now available \cite{M-I}, relevance of this problem is no longer obvious.

\subsection{Topology of the indexing set}
The goal of the present paper is to ``do everything vice versa'' slightly more carefully, so as to avoid running 
into set-theoretic troubles.
The trouble noted in Example \ref{ZFC} is really a very basic trouble with the $\lim^1$ functor (see 
\cite{M-III}*{Example \ref{lim:MP}}); but it exists only for inverse systems indexed by an uncountable set, 
or more precisely by a set with no countable cofinal subset.
It is therefore natural to look for a ``forgotten topology'' of the indexing set, and to simply amend the definition
of derived limits so as to take this topology into account.
This is precisely what we will do in the present paper.

The indexing set we are mostly interested in is the poset $K(X)$ of all nonempty compact subsets of 
a separable metrizable space $X$.
(The empty subset is excluded for technical reasons; its exclusion will be harmless.)
Given some metric on $X$, the usual Hausdorff metric (see \S\ref{Hausdorff metric}) on $K(X)$ makes it into 
a separable metric space, whose underlying topology (also known as the Vietoris topology) depends only on
that of $X$.
In fact, $K(X)$ turns out to be a {\it pospace} (=partially ordered space): its topology and order agree in 
the sense that the order relation is a closed set (Lemma \ref{Hausdorff pospace}).

As for open covers of $X$ or their nerves, unfortunately they do not seem to carry any natural separable 
metrizable topology.%
\footnote{Also, a technical disadvantage is that they do not form a poset; while irreducible open covers do, their
topology might be less natural.
One can, of course, work with preorders or filtered categories instead of posets, so this is not a serious problem.}
While a more elegant solution may exist, one obvious thing to do is to simply embed $X$ in some compact ANR,
for example the Hilbert cube $Q$, and look at its open neighborhoods there.
Their complements are precisely all compact subsets of $Q\but X$, and so the previous construction applies 
here as well.

\subsection{Topological order complex}
The usual derived limits are revisited in \S\ref{discrete} below, where their usual definition is reformulated in
a more geometric form.
Namely, the derived limits of an inverse system $D$ indexed by a poset $P$ can be understood as the cohomology 
groups of a certain sheaf, denoted $\Holim D$, over the order complex $|P|$.

Therefore our next goal is to define an appropriate topology on the order complex of a poset $P$ that is also 
endowed with a topology.

An obvious approach is to consider the geometric realization of the corresponding simplicial space, also known 
as the classifying space of the topological category (see \S\ref{wrong}).
However, the resulting topology can well be non-metrizable even when the original topology of $P$ is discrete, 
so it is ``wrong'' for our purposes.

The next obvious approach is to inject $|P|$ into some metrizable (perhaps, vector) space naturally ``spanned''
by $P$ and take the induced topology.
But it turns out that there is more than one good way to do so.

The main result of \S\ref{pospaces} is as follows (see Theorem \ref{HM=AE}).

\begin{maintheorem} \label{thA}
If $P$ is a pospace, the obvious injective maps of the order complex $|P|$ into the space of measurable 
functions $[0,1]\to P$ and into the space of probability measures on $P$ induce the same topology on the set $|P|$.
\end{maintheorem}

In \S\ref{simplicial hyperspace} we additionally prove the following (see Lemma \ref{simplicial hyperspace-def}(c) 
and Theorems \ref{simplicial-hyperspace}, \ref{topology3}).

\begin{maintheorem} The embedding $e$ of $K(X)$ into the space $C(X,\R)$ of real-valued functions on $X$,
given by $e(A)(x)=d(x,A)$, extends to an injective map of $|K(X)|$ in $C(X,\R)$ such that
the induced topology on $|K(X)|$ does not depend on the choice of the metric $d$ on $X$.
Moreover, it coincides with the two topologies of Theorem \ref{thA}.
\end{maintheorem}

\subsection{Derived limits over a pospace}

Let $P$ be a pospace and $\F$ be a {\it posheaf} (=partially ordered sheaf) of abelian groups over $P$
(see definition in \S\ref{posheaves}).
Then there is a certain sheaf $\Holim\F$ of abelian groups over the metrizable topological order complex $|P|$
(see definition in \S\ref{holim-definition}), and we define $\Lim^p\F$ to be $H^p(|P|;\,\Holim\F)$.
When the topology of $P$ is discrete, $\F$ boils down to a $P$-indexed inverse system of abelian groups, and 
$\Lim^p\F$ coincides with its usual $p$th derived limit.

\begin{maintheorem} \label{thC}
Let $X$ be a separable metrizable space. 
For each $q\ge 0$ there exists a posheaf $\F^q$ over $K(X)$ such that
the stalk $\F^q_{K_\alpha}\simeq H^q(K_\alpha)$ for each $K_\alpha\in K(X)$, and there exists a spectral sequence
of the form \[\Lim^p\F^q=E_2^{pq}\ \Rightarrow\ H^{p+q}(X).\]
\end{maintheorem}

\begin{maintheorem} \label{thD}
Let $X$ be a subset of $S^n$ and let $U(X)$ be the pospace of its open neighborhoods $\ne S^n$, ordered 
by inclusion and topologized by the Hausdorff metric.
For each $q\ge 0$ there exists a posheaf $\F_q$ over $U(X)$ such that the stalk 
$(\F_q)_{P_\beta}\simeq H_q(P_\beta)$ for each $P_\beta\in U(X)$, and there exists a 
spectral sequence of the form \[\Lim^p\F_q=E^2_{-p,q}\ \Rightarrow\ H_{q-p}(X).\]
\end{maintheorem}

Key ingredients of the proof of Theorem \ref{thC} (see Theorem \ref{lerayss}) are the Leray 
spectral sequence of a continuous map, a homotopy equivalence lemma (Theorem \ref{atoms}) and a lemma 
on compatibility of the $\Holim$ operator with Leray sheaves (Theorem \ref{Leray}(b)).
Theorem \ref{thD} (see Theorem \ref{lerayss2}) additionally employs the Sitnikov duality.

All theorems of the present paper are proved in ZFC.
Thus it can be argued, in view of Example \ref{ZFC}, that these results succeed to ``expel set theory from 
algebraic topology''.

\section{Background: Usual derived limits revisited} \label{discrete}

Let us now discuss in some detail the usual derived limits $\lim^p D$ for a diagram $D$ of abelian groups
indexed by a poset $P$.
If $P$ is viewed as a category%
\footnote{With elements of $P$ as objects, precisely one morphism $p\to q$ when $p\le q$
and no morphisms $p\to q$ when $p\not\le q$}, 
then such a diagram $D$ is simply a functor from $P$ to the category of abelian groups.
It will be notationally more convenient to regard $D$ as a set of data: groups $G_p$ for all $p\in P$ and
homomorphisms $\phi^p_q\:G_p\to G_q$ for all $p,q\in P$ with $p\le q$.

Let us define a new poset $G=\bigsqcup_{p\in P} G_p$ with $(p,g)\le (q,h)$ if and only if $p\le q$ and 
$\phi^p_q(g)=h$.
The map $\pi\:G\to P$ defined by $\pi(p,g)=p$ is monotone and so induces a simplicial map $\F\:|G|\to|P|$
between the order complexes%
\footnote{The order complex of a poset $P$ is the simplicial complex with elements of $P$ as vertices,
and with finite chains $p_1\le\dots\le p_n$ as simplexes.
We endow simplicial complexes with the metric topology.}.
The simplicial complex $|G|$ is known as the {\it homotopy colimit} $\hocolim D$, and we will denote 
the map $\F$ by $\Hocolim D$.

In fact, $\Hocolim D$ is a cosheaf (in the geometric sense, i.e.\ a ``complete spread'' of R. Fox,
or equivalently a ``display space'' of J. Funk, see \cite{Woo}*{Appendix B}), and we have
$H_0(|P|;\,\Hocolim D)\simeq\colim D$ and $H_i(|P|;\,\Hocolim D)=0$ for $i>0$.

Here are some details.
As a cosheaf of sets, $\F=\Hocolim D$ can be recovered as the cosheafafication of its precosheaf%
\footnote{A {\it precosheaf} of abelian groups on a topological space $X$ is a covariant functor from 
the poset of all open subsets of $X$ ordered by inclusion to the category of abelian groups.} 
of cosections $F$, defined by $F(U)=\pi_0\big(\F^{-1}(U)\big)$.
Namely, the costalks $\F_x=\F^{-1}(x)$ of $\F$ are the inverse limits $\lim_U F(U)$ over all open $U$
containing $x$, and $\F$ itself is the projection $\bigsqcup_{x\in |P|}\F_x\to |P|$, where the disjoint union
of sets is endowed with the topology with basis consisting of $V_{U,s}=\{(x,t)\mid p^x_U(t)=s\}$, where 
$p^x_U\:\F_x\to F(U)$ is the natural map from the inverse limit, $U$ runs over open subsets of $|P|$ 
and $s$ runs over $F(U)$.

Let us note that each $\F^{-1}(p\le q)$ is nothing but the mapping cylinder $\cyl(\phi^p_q)$.
In general, $\F^{-1}(p_1\le\dots\le p_n)$ is the iterated mapping cylinder 
$\cyl(\phi^{p_1}_{p_2},\dots,\phi^{p_{n-1}}_{p_n})$, that is the mapping cylinder of the composition 
$\cyl(\phi^{p_1}_{p_2},\dots,\phi^{p_{n-1}}_{p_{n-2}})\xr{\pi}
G_{p_{n-1}}\xr{\phi^{p_{n-1}}_{p_n}} G_{p_n}$, where $\pi$ is the natural projection.

The precosheaf $F$ is of a combinatorial type, known as a ``costack'' or 
a ``local coefficient system'': it is determined by its values on the open stars%
\footnote{The open star of a simplex $\sigma$ is the union of the interiors of all simplexes containing $\sigma$.}
of simplexes $\ost(\sigma)$, $\sigma\in |P|$.
So the homology of $|P|$ with coefficients in $\F$ is the homology of the chain complex $C_*(|P|;\,F)$
consisting of the groups $C_n=\bigoplus_{\sigma^n\in |P|}F_\sigma$, where $F_\sigma=F\big(\ost(\sigma)\big)$,
with the differential $\partial\:C_n\to C_{n-1}$ defined on each $g_\sigma\in F_\sigma$ by 
$\partial(g_\sigma)=\sum_{\tau^{n-1}\subset\sigma^n}[\tau:\sigma]F_{\sigma\tau}(g_\sigma)$, where 
$[\tau:\sigma]=\pm1$ is the incidence index and $F_{\sigma\tau}\:F_\sigma\to F_\tau$
is the image of $F$ on the inclusion $\ost(\sigma)\subset\ost(\tau)$.
In our case, clearly, $F_\sigma=G_{p_1}$ for $\sigma=(p_1\le\dots\le p_n)$ and
$F_{\sigma\tau}=\phi^{p_1}_{p_2}\:G_{p_1}\to G_{p_2}$ if $\tau=(p_2\le\dots\le p_n)$, whereas
in the $n-1$ other cases $F_{\sigma\tau}=\id\:G_{p_1}\to G_{p_1}$.
From this explicit description it is not hard to compute $H_i(|P|;\,\F)$.

In order to deal with $\lim$ and $\lim^p$ we need to dualize the above.
A well-known approach, going back to Bousfield and Kan \cite{BK}, is by considering the usual homotopy limit,
which is built out from the usual mapping cocylinders%
\footnote{The usual mapping cocylinder of $f\:X\to Y$ is the pullback of $f$ and the evaluation map $Y^I\to Y$, 
$p\mapsto p(1)$. 
Let us note that the usual mapping cylinder of $f$ is the pushout of $f$ and the inclusion $X\to X\x I$, 
$x\mapsto (x,1)$.}.
But it turns out that there is also a different approach, which is arguably more intuitive. 

Let $\Gamma$ be the coequalizer of
$\bigsqcup_{p\le q}G_p\x\ost(p\le q)\underset i{\overset j\rightrightarrows}\bigsqcup_p G_p\x\ost(p)$, 
that is, the quotient of the latter disjoint union by identifying the images of the former disjoint union 
under the two inclusions $i$ and $j$, which arise from the inclusions $i_{pq}\:\ost(p\le q)\subset\ost(p)$ and 
$j_{pq}\:\ost(p\le q)\subset\ost(q)$.
Since $i$ and $j$ commute with the projection to $|P|$, we get a map $\tilde\F\:\Gamma\to|P|$, which we will
denote by $\Holim D$ (breaking with the tradition of understanding $\holim D$ as constructed from path spaces).

It is easy to see that $\Holim D$ is a sheaf (in the geometric sense, i.e.\ the ``\'etal\'e space'')
--- normally a non-Hausdorff one, as is common for sheaves%
\footnote{Let us note that the coequalizer of the maps $(0,\infty)\x\Z\rightrightarrows\R\x\Z$ given by 
$(x,n)\mapsto(x,n)$ and $(x,n)\mapsto(x,n+1)$ is non-Hausdorff, but has a countable base of topology. 
In contrast, the coequalizer of the maps $[0,\infty)\x\Z\rightrightarrows\R\x\Z$ given by the same formulas
is Hausdorff, but does not have a countable base of neighborhoods at $0$.},
and it turns out, as we will see shortly, that $H^n(|P|;\,\Holim D)=\lim^n D$.

Let us note that each $\tilde\F^{-1}(p\le q)$ is homeomorphic to $\cocyl(\phi^p_q)$, where 
$\cocyl(X\xr{f}Y)$ denotes, unconventionally, the pushout of $f\x\id_{(0,1]}\:X\x(0,1]\to Y\x(0,1]$
and the inclusion $X\x(0,1]\subset X\x I$.
In general, $\tilde\F^{-1}(p_1\le\dots\le p_n)$ is the iterated mapping cocylinder 
$\cocyl(\phi^{p_1}_{p_2},\dots,\phi^{p_{n-1}}_{p_n})$, that is the mapping cocylinder of the composition 
$\cocyl(\phi^{p_1}_{p_2},\dots,\phi^{p_{n-2}}_{p_{n-1}})\xr{\pi}G_{p_{n-1}}\xr{\phi^{p_{n-1}}_{p_n}} G_{p_n}$, 
where $\pi$ is the natural projection.

There is another definition of the sheaf $\Holim D$, which may look fancier but is more categorical and
will be easier to generalize.
Let $P'$ be the poset of nonempty faces of the simplicial complex $|P|$; thus $P'$ is the set of all nonempty
finite chains of $P$, ordered by inclusion.
Sending every chain to its greatest element yields a monotone map $\rho\:P'\to P$.
Let $\rho^*F\:\rho^*G\to P'$ be the pullback of the monotone map $\pi\:G\to P$ along $\rho$.%
\footnote{Thus $\rho^*G$ consists of tuples $(p_1<\dots<p_n;\,g)$, where $g\in G_{p_n}$, and is ordered by 
$(p_1<\dots<p_n;\,g)\le (q_1<\dots<q_m;\,h)$ if and only if $(p_1<\dots<p_n)$ is a subchain of $(q_1<\dots<q_m)$ 
and $\phi^{p_n}_{q_m}(g)=h$.}
Since $\rho^*F$ is monotone, it may be viewed as a continuous map with respect to the Alexandrov topologies%
\footnote{The Alexandrov topology on the set of points of a poset $P$ consists of all subsets $U\subset P$
such that if $p\in U$ and $q\ge p$, then also $q\in U$.}, 
and as such it is easily seen to be a sheaf.%
\footnote{Indeed, the smallest open neighborhood of an element $(p_1<\dots<p_n;\,g)\in\rho^*G$ consists 
of all tuples $\big(q_1<\dots<q_m;\,h\big)\in\rho^*G$ such that $(p_1<\dots<p_n)$ is a subchain of $(q_1<\dots<q_m)$ 
and $h=f^{p_n}_{q_m}(g)$.
Since $h$ is uniquely determined by $(q_1<\dots<q_m)$, this neighborhood projects homeomorphically onto the smallest 
open neighborhood of $(p_1<\dots<p_n)$ in $P'$, which consists of all chains $(q_1<\dots<q_m)$ that have 
$(p_1<\dots<p_n)$ as a subchain.}
On the other hand, there is a continuous map $g\:|P|\to P'$ defined by sending the interior of every simplex
to that same simplex regarded as a point of $P'$. 
The pullback of $\rho^*F$ via $g$ is nothing but the sheaf $\tilde\F=\Holim D$.

The presheaf of sections $\tilde F$ of the sheaf $\tilde\F$ is again determined by its values on 
the open stars of simplexes, and so is a ``stack'' or a ``local coefficient system''.
So the cohomology of $|P|$ with coefficients in $\tilde\F$ is the cohomology of the cochain complex 
$C^*(|P|;\,\tilde F)$ consisting of the groups $C^n=\prod_{\sigma^n\in |P|}\tilde F_\sigma$, where 
$\tilde F_\sigma=\tilde F\big((\ost(\sigma)\big)$, with differential $\delta\:C^n\to C^{n+1}$ defined by
$(\delta c^n)(\tau^{n+1})=\sum_{\sigma^n\subset\tau^{n+1}}
[\sigma:\tau]\tilde F_{\tau\sigma}\big(c^n(\sigma)\big)$, 
where $\tilde F_{\tau\sigma}\:\tilde F_\sigma\to\tilde F_\tau$
is the image of $\tilde F$ on the inclusion $\ost(\tau)\subset\ost(\sigma)$.
In our case, clearly, $\tilde F_\sigma=G_{p_n}$ for $\sigma=(p_1\le\dots\le p_n)$ and
$\tilde F_{\sigma\tau}=\phi^{p_{n-1}}_{p_n}\:G_{p_{n-1}}\to G_{p_n}$ if $\tau=(p_1\le\dots\le p_{n-1})$, whereas
in the $n-1$ other cases $\tilde F_{\sigma\tau}=\id\:G_{p_n}\to G_{p_n}$.
From this explicit description it immediately follows that $H^n(|P|;\,\Holim D)=\lim^n D$.

Now let $\Delta$ be a diagram of simplicial complexes $X_p$ and simplicial maps $f^p_q\:X_p\to X_q$ indexed by 
the poset $P$.
We may again consider the map $\Hocolim\Delta\:\hocolim\Delta\to|P|$, which satisfies
$(\Hocolim\Delta)^{-1}(p_1\le\dots\le p_n)=\cyl(f^{p_1}_{p_2},\dots,f^{p_{n-1}}_{p_n})$.
On the other hand, for each $n$ we have a diagram $D_n$ of abelian groups $H_n(X_p)$ and homomorphisms
$(f^p_q)_*\:H_n(X_p)\to H_n(X_q)$ indexed by $P$, and also a diagram $D^n$ of abelian groups $H^n(X_p)$ and 
homomorphisms $(f^p_q)^*\:H^n(X_q)\to H^n(X_p)$ indexed by the dual poset $P^*$ (i.e.\ the same set with 
reversed order).
It is not hard to see that the Leray sheaf $\h^n(\Hocolim\Delta)\simeq\Holim D^n$, and the Leray cosheaf
$\h_n(\Hocolim\Delta)\simeq\Hocolim D_n$.
Consequently, we have a spectral sequence
\[E_2^{pq}=\lim\nolimits^p D^q\simeq H^p(|P|;\,\Holim D^q)\Rightarrow H^{p+q}(\hocolim\Delta),\]
which is known as the Bousfield--Kan/Araki--Yoshimura spectral sequence \cite{BK}, \cite{AY}, and also a spectral sequence 
$H_p(|P|;\,\Hocolim D_q)\Rightarrow H_{p+q}(\hocolim\Delta)$, which collapses to an isomorphism 
$\lim D_n\simeq H_n(\hocolim\Delta)$.

It should be noted that if $(X_p,f^p_q)$ is a directed system of pointed subcomplexes of a simplicial complex $X$
and their inclusions (directed means that every two are both contained in a third one), then $\hocolim\Delta$ is 
homotopy equivalent to $X$ \cite{AY}, \cite{BK} (see also \cite{Se}, \cite{Vo}, \cite{DI}).
In this case the Bousfield--Kan/Araki--Yoshimura spectral sequence has $E_3^{pq}=0$ for $p\ge 2$ and in fact 
reduces to a Mdzinarishvili-type long exact sequence \cite{Oh}.
If moreover the subcomplexes $X_p$ are finite, then already $E_2^{pq}=\lim\nolimits^p D^q=0$ for $p\ge 2$ 
(see references in \cite{M-III}*{Remark \ref{lim:lim2fg}}) and consequently the spectral sequence reduces to
a Milnor-type short exact sequence.

\section{Partially ordered spaces}\label{pospaces}

By a {\it topological poset} we mean a topological space $P$ that is also a poset.
This is a topological category (in the sense of Segal \cite{Se}) with precisely one morphism $p\to q$
whenever $p\le q$ and no morphisms $p\to q$ otherwise, and with the set of morphisms topologized
as a subspace of $P\x P$.
(That it is a topological category means that the four structure maps are automatically continuous: 
the ``identity'' map from objects to morphisms, the ``source'' and ``target'' maps from morphisms to 
objects, and the ``composition'' map from morphisms squared to morphisms.)
The {\it dual} topological poset $P^*$ is the same space with the reversed order.

A {\it pospace} (=partially ordered space) is a topological poset such that the order relation $\le$,
viewed as a subset of $P\x P$, is closed in $P\x P$ with respect to the product topology.
When $P$ is metrizable, this is equivalent to saying that if $x_n\le y_n$ for each $n\in\N$ and
$x_n\to x$ and $y_n\to y$ as $n\to\infty$, then $x\le y$.

If $P$ is a topological poset, let $|P|$ denote the set of all formal sums $x=\sum\lambda_{i=1}^n x_i$, where 
$n\in\N$, $x_1\le\dots\le x_n$ is a nonempty finite chain in $P$, each $\lambda_i\ge 0$ and $\sum_{i=1}^n\lambda_i=1$.
We will now consider some natural topologies on $|P|$.

\subsection{Wrong construction (weak topology)} \label{wrong}

One natural topology on $|P|$ is well-known.
Namely, let $|P|_w$ be the classifying space of $P$ (or, in another terminology, the nerve of $P$), where $P$ 
is viewed as a topological category \cite{Se}.
Unfortunately, the topology of $|P|_w$ is ``wrong'' (for our purposes) in that it generally fails 
to be metrizable when $P$ is metrizable.
In particular, $|P|_w$ is non-metrizable whenever $P$ is discrete (as a space) and the order complex of $P$ 
(as a poset) is not locally finite.
But actually a satisfactory ``correction'' of the weak topology is already known in this particular case 
(see Remark \ref{quot-unif} below).

Let us briefly review the construction of $|P|_w$.
By definition, $|P|_w$ is the geometric realization of the simplicial space (=simplicial object in the category of
topological spaces, cf.\ \cite{Fri}) where every $(n-1)$-simplex is a chain $p_1\le\dots\le p_n$ of $n$ elements of $P$.
Such a chain may be regarded as a monotone map $C\:[n]\to P$ (possibly non-injective), where $[n]$ denotes 
the $n$-element poset $\{1,\dots,n\}$ with the usual (total) order.
All such chains of length $n$ (possibly with repeats) form a subspace $P^{[n]}$ of the $n$-fold 
Cartesian product $P\x\dots\x P$ of spaces.

For example, the partially ordered space $[0,1]$ (the closed unit interval of the real line with its usual total
order) yields the ``standard skew $n$-simplex'' 
$[0,1]^{[n]}=\{(x_1,\dots,x_n)\mid 0\le x_1\le\dots\le x_n\le 1\}\subset\R^n$.
Let us note that we already have to consider chains with repeats here in order to get the closed simplex.

Every monotone map $f\:[m]\to[n]$ induces a continuous map $f^*\:P^{[n]}\to P^{[m]}$ by $C\mapsto Cf$.
In particular, we have $f^*\:[0,1]^{[n]}\to [0,1]^{[m]}$.
On the other hand, the ``standard symmetric $(n-1)$-simplex''
$\Delta^{n-1}=\{(x_0,\dots,x_n)\in\R^n\mid x_0,\dots,x_n\ge 0,\,x_0+\dots+x_n=1\}$
behaves covariantly: $f$ induces a map $f_*\:\R^m\to\R^n$ taking $\Delta^{m-1}$ into $\Delta^{n-1}$.

With this notation, $|P|_w$ is the quotient space of $\bigsqcup_{n\in\N} P^{[n]}\x\Delta^{n-1}$ 
by the equivalence relation $(C,f_*(t))\sim(f^*(C),t)$ for every monotone map $f\:[m]\to[n]$.

Let us note that if we consider only monotone injections so as to avoid degenerate simplexes, this will give 
a completely unintended topology on $|P|$, because chains with repeats arise naturally as limits of chains 
without repeats.

\begin{remark} \label{quot-unif}
One can attempt to ``correct'' the weak topology.
Let us fix a metric on $P$ that is bounded above by $1$ and metrize $\Delta^{n-1}$ by means of the $l_\infty$ metric 
on $[0,1]^{[n-1]}$.
Let us use the $l_\infty$ product metric on each summand of $\bigsqcup_{n\in\N} P^{[n]}\x\Delta^{n-1}$ (note that
it will be bounded above by $1$) and set the distance between any pair of points in distinct summands equal to $1$.
This determines a metric, and in particular a metrizable uniform structure, on the disjoint union (beware that it is 
not the uniformity of the disjoint union, which is generally non-metrizable).
Let $|P|_{uw}$ be the quotient endowed with the topology of the quotient uniformity (not to be confused
with the quotient topology).

If $P$ is discrete (as a space), and is metrized by setting every distance equal to $1$, then $|P|_{uw}$ 
is metrizable \cite{M3}.
In the general case, we have three open questions: 
\begin{enumerate}
\item Is $|P|_{uw}$ well defined (i.e.\ independent of the choice of metric on $P$)? 
\item Is $|P|_{uw}$ metrizable? 
\item Is the topology of $|P|_{uw}$ same as the topology of $|P|_{HM}$ and $|P|_{AE}$ (see below)?
\end{enumerate}
\end{remark}

\subsection{Hartman--Mycielski construction (measurable functions)} 

We will refer to \cite{M2}*{\S\ref{metr:measurable functions}} for a detailed discussion of spaces of
measurable functions, and we will use some notation introduced there.

Let $P$ be a topological poset.
Given an $x=\sum_{i=1}^n\lambda_i x_i\in |P|$, where $x_1\le\dots\le x_n$, each $\lambda_i\ge 0$ and 
$\sum_{i=1}^n\lambda_i=1$, let us define a step function $\phi_x\:[0,1)\to P$ by $\phi_x(t)=x_k$ if $t\in I_k$, 
where $I_k=[\sum_{i<k}\lambda_i,\sum_{i\le k}\lambda_i)$.
Let $|P|_{HM}\subset HM(P)$ be the set of all such functions, endowed with the topology of convergence in
measure.
If $P$ is endowed with a metric $d$, we further endow $|P|_{HM}$ with the metric $L_1$ (which induces
the same topology of uniform convergence).

\begin{example} Let $P=\{0,\dots,n\}$ with the usual order and with the usual metric (i.e., a subpospace of $\R$).
Then $|P|_{HM}$ is isometric to the standard skew $n$-simplex 
$\Delta=\{(t_1,\dots,t_n)\in\R^n\mid 0\le t_1\le\dots\le t_n\le 1\}$ with the $l_1$ metric.

Indeed, given step functions $f,g\:[0,1)\to P$ with $f^{-1}(i)=[t_i,t_{i+1})$ and $g^{-1}(i)=[s_i,s_{i+1})$, 
where $0=t_0\le\dots\le t_{n+1}=1$ and $0=s_0\le\dots\le s_{n+1}=1$, they correspond to 
$T=(t_1,\dots,t_n)\in\Delta$ and $S=(s_1,\dots,s_n)\in\Delta$, with $l_1(T,S)=\sum_{i=1}^n|t_i-s_i|$.

On the other hand, we have $f=f_1+\dots+f_n$ as elements of the vector space of measurable functions $I\to\R$, 
where each $f_i\:I\to P$ is a two-valued step function, $f_i(t)=0$ if $t<t_i$ and $f_i(t)=1$ if $t\ge t_i$.
Similarly, $g=g_1+\dots+g_n$, where $g_i(t)=0$ if $t<s_i$ and $g_i(t)=1$ if $t\ge s_i$.
If $f_i(t)>g_i(t)$, then $t_i\le t< s_i$; and if additionally $f_j(t)<g_j(t)$, then $s_j\le t< t_j$.
Hence $s_j<s_i$ and $t_i<t_j$, which is a contradiction.
Therefore either $f_i(t)\ge g_i(t)$ for all $i$, or $g_i(t)\ge f_i(t)$ for all $i$.
Thus $L_1(f,g)=\int_I|f(t)-g(t)|\,dt=\int_I\sum_{i=1}^n |f_i(t)-g_i(t)|\,dt=\sum_{i=1}^n L_1(f_i,g_i)=
\sum_{i=1}^n|t_i-s_i|=l_1(T,S)$.
\end{example}

\subsection{Resolution} \label{resolution}

If $P$ is a poset, let $A(P)$ be its set of atoms (that is, elements $p\in P$ such that $q\not<p$ for each $q\in P$).
If $P$ is a topological poset, $A(P)$ is a topological subposet of $P$, with discrete order (i.e.\ no elements
are comparable).
In particular, $|A(P)|$ is homeomorphic to $A(P)$.

\begin{example} If $X$ is a metrizable space and $K(X)$ is the pospace of its nonempty compact subsets (see Lemma
\ref{Hausdorff pospace}), then $A\big(K(X)\big)=X$.
\end{example}

If $P$ and $Q$ are posets, $P\x Q$ is a poset, where $(p,q)\le (p',q')$ iff $p\le p'$ and $q\le q'$.
If $P$ and $Q$ are topological posets, $P\x Q$ is a topological poset.
In particular, $A(P)\x P$ is a topological poset, where $(a,p)\le (b,q)$ iff $a=b$ and $p\le q$.
It is easy to see that $|A(P)\x P|$ is homeomorphic to $A(P)\x|P|$.

Let $E(P)$ be the topological subposet of $A(P)\x P$ consisting of all pairs $(a,p)$ such that $a\le p$.
Thus $|E(P)|$ is a subspace of $|A(P)\x P|$, which is identified with the subspace of $A(P)\x |P|$
consisting of all pairs $\big(a,\,\sum_{i=1}^n\lambda_i x_i\big)$ such that $a\le x_1$ (and also 
$x_1\le\dots\le x_n$, each $\lambda_i\ge 0$ and $\sum_{i=1}^n\lambda_i=1$).

\begin{theorem} \label{atoms} $|E(P)|$ is homotopy equivalent to $A(P)$.
\end{theorem}

We will prove and use this theorem only for metrizable $P$, but the proof straightforwardly extends 
to the non-metrizable case.

\begin{proof} Let $\pi\:A(P)\x P\to A(P)$ be the projection and let 
$\iota\:A(P)\to E(P)\subset A(P)\x P$ be defined by $\iota(a)=(a,a)$.
Clearly, $\pi\iota=\id_{A(P)}$, so in particular, $\iota$ is an embedding of $A(P)$ into $E(P)\subset |E(P)|$.
Let us show that $|E(P)|$ deformation retracts onto the image of $\iota$.
More precisely, we will construct a homotopy $h_t$ between $\id_{|E(P)|}$ and the composition
$|E(P)|\subset A(P)\x|P|\xr{\pi} A(P)\xr{\iota} E(P)\subset |E(P)|$.

Every pair $\big(a,\,\sum_{i=1}^n\lambda_i x_i\big)\in|E(P)|$, with $a\le x_1\le\dots\le x_n$, can also be 
written as $\big(a,\sum_{i=0}^n\lambda_i x_i\big)$, where $x_0=a$ and $\lambda_0=0$.
Thus we may assume without loss of generality that $a=x_1$.
Let $X=\sum_{i=1}^n\lambda_i x_i$ and let us define $h_t(x_1,X)=(x_1,X_t)$, where $X_t=tx_1+(1-t)X$.
Then $h_0=\id_{|E(P)|}$ and $h_1=\iota\pi$.
It remains to check that $h_t$ is continuous.

Let $Y=\sum_{i=1}^m\mu_i y_i$ and let $Y_t=ty_1+(1-t)Y$.
Then $L_1(X_t,Y_t)=td(x_1,y_1)+(1-t)L_1(X,Y)$.
If we use the $l_1$ product metric on the product $A(P)\x |P|$, then
$d\big(h_t(x_1,X),\,h_t(y_1,Y)\big)=d(x_1,y_1)+L_1(X_t,Y_t)\le 2d(x_1,y_1)+L_1(X,Y)$.
Also, we have $d\big(h_t(x_1,X),\,h_s(x_1,X)\big)=L_1(X_s,X_t)\le n|s-t|$.
It follows that $h_t$ is continuous.
\end{proof}

\subsection{Arens--Eels construction (probability measures)} \label{arens-eels}

A {\it finite measure} on a space $X$ is a function $\mu\:X\to\R$ with support in finitely many points.
Thus $\mu=\sum_{i=1}^n\mu_i\delta_{x_i}$, where $\mu_i=\mu(x_i)$ and $\delta_x\:X\to\R$ is the Dirac measure,
defined by $\delta_x(y)=1$ if $x=y$ and $0$ if $x\ne y$.
Let $\mu(X)=\sum_{x\in X}\mu(x)$, that is, $\mu(X)=\sum_{i=1}^n\mu_i$ in the previous notation.
If $\mu(X)=1$, then $\mu$ is called a {\it finite probability measure}.
Let $AE(X)$ denote the set of all finite measures on $X$ and $AE_V(X)$ its subset consisting of all $\mu$
with $\mu(X)=V$.
Clearly, $AE(X)$ is a real vector space and $AE_V(X)$ is an affine hyperplane in $AE(X)$ for each $V\in\R$.
Also, $AE_0(X)$ is vector subspace of $AE(X)$.

Let us fix a metric $d$ on $X$.
If $\nu\in AE_0(X)$, let $||\nu||=\inf\sum_{k=1}^r|\nu_k| d(x_k,y_k)$,
where the infimum is over all representations $\nu=\sum_{k=1}^r\nu_k(\delta_{x_k}-\delta_{y_k})$.
Such representations exist since $\nu(X)=0$.

If $\lambda,\mu\in AE_V(X)$, let $\rho(\lambda,\mu)=||\lambda-\mu||$.
In particular, we get a metric $\rho$ on the set $AE_1(X)$ of all finite probability measures, which is
called the {\it Kantorovich metric} or alternatively the Wasserstein (Vasershtein) metric.
If we understand $\mu$ and $\nu$ as distributions of masses, $\rho(\lambda,\mu)$ can be interpreted as 
the minimal amount of work needed to transport $\lambda$ into $\mu$.

Now let $P$ be a metrizable topological poset.
Given a formal sum $x=\sum_{i=1}^n\lambda_i x_i\in |P|$, we have the finite probability measure 
$\mu_x\bydef \sum\lambda_i\delta_{x_i}$.
Let $|P|_{AE}\subset AE_1(P)$ be the set of all such measures, endowed with the Kantorovich metric $\rho$.

\begin{example} Let $P=\{x_1,\dots,x_n\}$ linearly ordered by $x_i\le x_j$ iff $i\le j$ and with $d(x_i,x_j)=1$ 
whenever $i\ne j$.
Then $|P|_{AE}$ is isometric to a homoteth of the standard $(n-1)$-simplex 
$\Delta=\{(\lambda_1,\dots,\lambda_n)\in\R^n\mid \lambda_k\ge 0,\,\sum_{k=1}^n\lambda_k=1\}$ with the $l_1$ metric.

Indeed, given $x=\sum_{k=1}^n\lambda_k x_k\in |P|$ and $y=\sum_{k=1}^n\mu_k y_k\in |P|$, they correspond to
$\lambda=(\lambda_1,\dots,\lambda_n)\in\Delta$ and $\mu=(\mu_1,\dots,\mu_n)\in\Delta$, with
$l_1(\lambda,\mu)=\sum_{k=1}^n|\lambda_k-\mu_k|$.

On the other hand, let $S^+=\{k\in [n]\mid \lambda_k>\mu_k\}$ and $S^-=\{k\in[n]\mid\lambda_k<\mu_k\}$.
Also let $N^+=\sum_{i\in S^+}(\lambda_i-\mu_i)$ and $N^-=\sum_{j\in S_-}(\mu_j-\lambda_j)$.
Then $N^+-N^-=\sum_{k=1}^n(\lambda_k-\mu_k)=1-1=0$ and $N^++N^-=\sum_{k=1}^n|\lambda_k-\mu_k|=l_1(\lambda,\mu)$.
Hence $N^+=l_1(\lambda,\mu)/2$.

Now we have $\rho(\mu_x,\mu_y)=\sum_{i\in S^+}\sum_{j\in S^-}\nu_{ij}d(x_i,x_j)$, where each $\nu_{ij}>0$ and
$\mu_x-\mu_y=\sum_{i\in S^+}\sum_{j\in S^-}\nu_{ij}(\delta_{x_i}-\delta_{x_j})$.
Then $\lambda_i-\mu_i=\sum_{j\in S^-}\nu_{ij}$ for each $i\in S^+$ and
$\mu_j-\lambda_j=\sum_{i\in S^+}\nu_{ij}$ for each $j\in S^-$.
Hence $\rho(\mu_x,\mu_y)=\sum_{i\in S^+}\sum_{j\in S^-}\nu_{ij}=N^+=l_1(\lambda,\mu)/2$.
\end{example}

\subsection{Comparison} 

\begin{example}
The bijection $\Phi\:|P|_{AE}\to|P|_{HM}$, $\mu_x\mapsto f_x$, is not uniformly continuous in general, 
even if $P$ is a pospace with disrete uniform structure.
Indeed, let $P=\{x_i\mid i\in\N\}$ be set of natural numbers with its usual linear order but with elements
denoted by $x_0,x_1,\dots$ instead of $0,1,\dots$, and with $d(x_i,x_j)=1$ for every two distinct $i,j\in\N$.
Given an $n\in\N$, let $p_n=\frac1n x_0+\dots+\frac1n x_{n-1}$ and $q_n=\frac1n x_1+\dots+\frac1n x_n$.
Then $\rho(\mu_{p_n},\mu_{q_n})=\frac1n d(x_0,x_n)=\frac1n$.
On the other hand, $L_1(f_{p_n},f_{q_n})=\frac1n d(x_0,x_1)+\dots+\frac1n d(x_{n-1},x_n)=n\frac1n=1$.
\end{example}

\begin{remark} The deformation retraction in Theorem \ref{atoms} is uniformly continuous 
(in fact, Lipschitz) with respect to the Kantorovich metric. 
Indeed, in the notation of the proof of Theorem \ref{atoms}, $\rho(X_s,X_t)\le |s-t|$.
(Also, $\rho(X_t,Y_t)=td(x_1,y_1)+(1-t)\rho(X,Y)$.)
\end{remark}

\begin{theorem} \label{HM=AE}
If $P$ is a metrizable pospace, then $|P|_{HM}$ and $|P|_{AE}$ are homeomorphic.
\end{theorem}

\begin{proof} Let us fix some metric on $P$.
Clearly, $\rho(\mu_x,\mu_y)\le L_1(f_x,f_y)$.

It remains to show that $\Phi\:|P|_{AE}\to|P|_{HM}$, $\mu_x\mapsto f_x$, is continuous.
We may assume that $P$ has diameter $\le 1$. 
Let $\eps>0$.
Let $x=\sum_{i=1}^n\lambda_i x_i$, where $x_1<\dots<x_n$, each $\lambda_i\ge 0$ and $\sum_{i=1}^n\lambda_i=1$.
Since $x_1\not\ge\dots\not\ge x_n$ and $\not\ge$ is open as a subset of $P\x P$, there exist pairwise disjoint
neighborhoods $U_1,\dots,U_n$ of $x_1,\dots,x_n$ such that $y_1\not\ge\dots\not\ge y_n$ whenever each $y_i\in U_n$.
Let $\delta>0$ be such that $(2n+1)(n+1)\delta<\eps$, each $\lambda_i>\delta$, and each $U_i$ contains the ball 
of radius $\delta$ about $x_i$.
Let $y=\sum_{i=1}^m\mu_i y_i$, where $y_1<\dots<y_m$, each $\mu_i\ge 0$ and $\sum_{i=1}^m\mu_i=1$,
and $\rho(\mu_x,\mu_y)<\delta^2$.

Then there exists a representation $\mu_x-\mu_y=
\sum_{j=1}^r\nu_j(\delta_{x_{m_j}}-\delta_{y_{n_j}})$ for some $r$, $m_j$ and $n_j$ such that
each $\nu_j\ge 0$ and $\sum_{j=1}^r\nu_j d(x_{m_j},y_{n_j})\le\delta^2$.
Let us write $x_{m_j}=X_j$ and $y_{n_j}=Y_j$.
Thus $\mu_x-\mu_y=\sum_{j=1}^r\nu_j(\delta_{X_j}-\delta_{Y_j})$ and $\sum_{j=1}^r\nu_j d(X_j,Y_j)\le\delta^2$.
Without loss of generality, $X_1\le\dots\le X_r$, and if $X_p=X_{p+1}$, 
then $Y_p<Y_{p+1}$.
Suppose that $X_p=X_{p+1}=\dots=X_q$, where either $p=1$ or $X_{p-1}<X_p$, and either $q=r$ or $X_q<X_{q+1}$.
Let us note that $m_p=m_{p+1}=\dots=m_q$.

Let $S_{m_p}=\{j\mid p\le j\le q,\,d(X_j,Y_j)\le\delta\}$ and $T_{m_p}=\{j\mid p\le j\le q,\,d(X_j,Y_j)>\delta\}$.
Then $\sum_{j\in T_i}\nu_j\le\sum_{j\in T_i}\nu_j d(X_j,Y_j)/\delta\le
\sum_{j=1}^r\nu_j d(X_j,Y_j)/\delta\le\delta$.
Since $\sum_{j=p}^q\nu_j=\lambda_{m_p}>\delta$, we have $S_{m_p}\ne\emptyset$.
Let $k_i=\min S_i$ and $l_i=\max S_i$.
Let us note that $m_{k_i}=m_{l_i}=i$.
Since $d(X_{l_i},Y_{l_i})\le\delta$ and $d(X_{k_{i+1}},Y_{k_{i+1}})\le\delta$, we have $Y_{l_i}\in U_{m_{l_i}}=U_i$
and $Y_{k_{i+1}}\in Y_{m_{k_{i+1}}}=U_{i+1}$.
Hence $Y_{l_i}\not\ge Y_{k_{i+1}}$, and therefore $Y_{l_i}<Y_{k_{i+1}}$ for each $i<n$.
Thus we have $Y_{k_1}<\dots<Y_{l_1}<Y_{k_2}<\dots<Y_{l_2}<Y_{k_3}<\dots$.
In other words, if $S=S_1\cup\dots\cup S_n$, then $Y_i<Y_j$ whenever $i,j\in S$ and $i<j$.
On the other hand, if $T=T_1\cup\dots\cup T_n$, then $\sum_{j\in T}\nu_j=\sum_{i=1}^n\sum_{j\in T_i}\nu_j\le n\delta$. 

Let $I_k=\big[\sum_{j<k}\nu_j,\sum_{j\le k}\nu_j\big)$.
Let us define a step function $f'_y\:[0,1)\to X$ by $f'_y(t)=Y_j$ if $t\in I_j$.
Let $I_S=\bigcup_{j\in S}I_j$ and $I_T=\bigcup_{j\in T}I_j$.
Then $d\big(f_x(t),f_y'(t)\big)\le\delta$ for each $t\in I_S$ and $\mu(I_T)=\sum_{j\in T}\nu_j\le n\delta$.
Hence $L_1(f_x,f'_y)=\int_I d\big(f_x(t),f'_y(t)\big)\,dt=\int_{I_S} d\big(f_x(t),f'_y(t)\big)\,dt+
\int_{I_T} d\big(f_x(t),f'_y(t)\big)\,dt\le \delta+n\delta=(n+1)\delta$,
using that $\mu(I_S)\le 1$ and $d\big(f_x(t),f_y'(t)\big)\le 1$ for each $t\in I_T$.
On the other hand, since $Y_i<Y_j$ whenever $i,j\in S$ and $i<j$, we have $f_y(t)=f'_y(t)$ unless $y$ belongs to
the $\mu(I_T)$-neighborhood of the set $\big\{\sum_{i\le k}\lambda_i\mid k=0,\dots,n\big\}$.
Hence $L_1(f_y,f'_y)\le 2(n+1)\mu(I_T)\le 2n(n+1)\delta$.
Thus $L_1(f_x,f_y)\le L_1(f_x,f'_y)+L_1(f'_y,f_y)\le (2n+1)(n+1)\delta\le\eps$.
\end{proof}

\section{Simplicial hyperspace} \label{simplicial hyperspace}

\subsection{Hausdorff metric} \label{Hausdorff metric}

If $X$ is a metric space, the hyperspace $K(X)$ of its nonempty compact subsets is endowed with 
the Hausdorff metric
\[d(A,B)=\max\big(\sup_{a\in A}d(a,B),\,\sup_{b\in B}d(A,b)\big),\] where 
$d(a,B)=d(B,a)=\inf_{b\in B}d(a,b)$.
Clearly, $X$ isometrically embeds in $K(X)$ via $x\mapsto\{x\}$.

Apart from being a metric space, $K(X)$ is also a poset by inclusion.

\begin{lemma} \label{Hausdorff pospace} $K(X)$ is a pospace.
\end{lemma}

\begin{proof} Let us show that $\not\le$ is open in $K(X)\x K(X)$. 
Suppose $A,B\in K(X)$, $A\not\ge B$.
Thus $B$ contains a point $b\notin A$. 
Let $\eps=d(A,b)/3$.
If $B'$ is $\eps$-close to $B$, then $B'$ contains a point $b'$ such that $d(b,b')\le\eps$. 
Hence $d(A,b')\ge 3\eps-\eps>\eps$.
Therefore if $A'$ is $\eps$-close to $A$, then $b'\notin A'$.
Hence $A'\not\ge B'$.
\end{proof}

The following lemma is well-known.

\begin{lemma} {\rm (See \cite{Bee2}.)} \label{Hausdorff} Let $A$ and $B$ be nonempty subsets of a metric space $X$.
Then

(a) $\sup_{a\in A}d(a,B)=\sup_{x\in X} d(x,B)-d(x,A)$;

(b) $d(A,B)=\sup_{x\in X}|d(x,A)-d(x,B)|$.
\end{lemma}

For the reader's convenience, we recall the proof for the case where $A$ and $B$ are compact
(the proof of the general case is only slightly different, but we do not need it).

\begin{proof}[Proof. (a)] Since $d(a,B)=d(a,B)-d(a,A)$, we have the $\le$ inequality.

To prove the $\ge$ inequality, is suffices to show that $d(x,B)-d(x,A)\le\sup_{a\in A}d(a,B)$
for each $x\in X$.
Since $A$ is compact, $d(x,A)=d(x,\alpha)$ for some $\alpha\in A$, and since $B$ is compact,
$d(\alpha,B)=d(\alpha,\beta)$ for some $\beta\in B$.
Then $d(x,B)\le d(x,\beta)\le d(x,\alpha)+d(\alpha,\beta)=d(x,A)+d(\alpha,B)$, and hence
$d(x,B)-d(x,A)\le d(\alpha,B)\le\sup_{a\in A} d(a,B)$. 
\end{proof}

\begin{proof}[(b)] This follows immediately from (a).
\end{proof}

Let $X$ be a metric space of diameter $\le 1$. 
By Lemma \ref{Hausdorff}(b) $K(X)$ admits an isometric embedding $e$ into the vector space $C_b(X)$ of 
bounded continuous functions $f\:X\to\R$ with the norm $||f||=\sup_{x\in X}|f(x)|$, defined by $e(A)(x)=d(x,A)$ 
for each nonempty compact $A\subset X$.
In fact, the image of $e$ lies in the convex subset of $C_b(X)$ consisting of 1-Lipschitz (in particular,
uniformly continuous) functions $X\to [0,1]$.

It is well-known that the image of the composition $X\subset K(X)\xr{e}C_b(X)$ is a linearly independent set
\cite{BlKl}.

\subsection{Simplicial hyperspace of a metric space}

Given a nonempty finite chain in $K(X)$, that is, a monotone map $C\:[n]\to K(X)$ from the totally ordered set 
$[n]=\{1,\dots,n\}$, $n\ge 1$, let us write $C_i=C(i)$, so that $C_1\subset\dots\subset C_n$, and let $|C|$ 
denote the convex hull of $e\big(C([n])\big)=\{e(C_1),\dots,e(C_n)\}$ in $C_b(X)$.
The following lemma guarantees that the convex hulls of two injective chains intersect precisely along 
the convex hull of their maximal common subchain.

\begin{lemma} \label{simplicial hyperspace-def}
Let $A\:[n]\to K(X)$ and $B\:[m]\to K(X)$ be injective monotone maps.

(a) Let $f_i=e(A_i)$ and $g_j=e(B_j)$, and suppose that $A_n\ne X$ if $n>0$ and $B_m\ne X$ if $m>0$.
If $\sum_{i=1}^n\lambda_i f_i=\sum_{j=1}^m\mu_j g_j$, where each $\lambda_i>0$ and each $\mu_j>0$, then $m=n$, 
each $f_i=g_i$ and each $\lambda_i=\mu_i$.

(b) The simplexes $|A|$ and $|B|$ either coincide or have disjoint interiors.

(c) $|A|\cap|B|=|C|$, where $C\:[k]\to K(X)$ is the pullback of $A$ and $B$ (that is, $C([k])=A([n])\cap B([m])$
and $C$ is injective).
\end{lemma}

\begin{proof}[Proof. (a)] Arguing by induction, we may assume that the assertion is known if $n$ or $m$ or 
both are replaced by smaller numbers.
Let $F=\sum_{i=1}^n\lambda_i f_i$ and $G=\sum_{j=1}^m\mu_j g_j$.
If $m>0$, then $B_m\ne X$, hence $G$ is not identically zero. 
Therefore so is $F$ and thus $n>0$.
Similarly, $n>0$ implies $m>0$.
Thus we may assume that $n>0$ and $m>0$.

Since each $\lambda_i>0$ and each $f_i(x)\ge 0$ for each $x\in X$, we have $F(x)=0$ if and only if 
$f_1(x)=\dots=f_n(x)=0$, i.e., $x\in A_1$.
Similarly, $G(x)=0$ if and only if $x\in B_1$.
Hence $A_1=B_1$ and so $f_1=g_1$.
Since $A$ is injective, there exists an $x_1\in A_2\but A_1$, and we have $F(x_1)=\lambda_1f_1(x_1)$
and $G(x_1)\ge\mu_1g_1(x_1)=\mu_1f_1(x_1)$.
Since $f_1(x_1)>0$, we get $\lambda_1\ge\mu_1$.
Similarly, $\lambda_1\le\mu_1$, and thus in fact $\lambda_1=\mu_1$.
Then $\sum_{i=2}^n\lambda_i f_i=\sum_{j=2}^m\mu_i g_i$, and by the induction hypothesis $m=n$, each $f_i=g_i$
for $i\ge 2$ and each $\lambda_i=\mu_i$ for $i\ge 2$.
\end{proof}

\begin{proof}[(b)] If $|A|$ and $|B|$ have non-disjoint interiors, then 
$\sum_{i=1}^n\lambda_i f_i=\sum_{j=1}^m\mu_j g_j$ for some $\lambda_i>0$ and $\mu_j>0$ such that
$\sum_{i=1}^n\lambda_i=1=\sum_{j=1}^m\mu_j$.
Let $N$ be the maximal number such that $A_N\ne X$ (so it must be either $n$ or $n-1$)
and let $M$ be the maximal number such that $B_M\ne X$ (so it must be either $m$ or $m-1$).
Then $\sum_{i=1}^N\lambda_i f_i=\sum_{j=1}^M\mu_j g_j$ and by (a) we have $M=N$,
each $f_i=g_i$ for $i\le M$ and each $\lambda_i=\mu_i$ for $i\le M$.
If $m=n=M$, then $|A|=|B|$ and we are done.
If $m=M+1$, then $1-(\mu_1+\dots+\mu_M)=1-(\lambda_1+\dots+\lambda_M)=\lambda_m>0$,
so $n=m$ and $\mu_m=1-(\mu_1+\dots+\mu_M)=\lambda_m$.
Also $f_m=g_m=e(X)$ and so we again have $|A|=|B|$.
The case $n=M+1$ is similar.
\end{proof}

\begin{proof}[(c)] This is standard.
Trivially $|C|\subset|A|\cap |B|$.
If $x\in|A|\cap|B|$, then $x$ lies in the interiors of $|A'|$ and $|B'|$ for some injective subchains 
$A'\:[n']\to[n]\xr{A}K(X)$ and $B'\:[m']\to[m]\xr{B}K(X)$.
Then by (b), $|A'|=|B'|$. 
Hence $|A'|\subset|C|$ and so $x\in |C|$.
\end{proof}

\subsection{Examples}

The following series of examples, which is not used in the sequel, analyzes the metric on the convex hull
in $C_b(X)$ of an individual chain in $K(X)$.

\begin{example} \label{discrete-hull}
(a) Let us consider the finite metric space $X_n=\{a_1,\dots,a_n\}$ with $d(a_i,a_j)=1$ for $i\ne j$. 
Let $A_i=\{a_1,\dots,a_i\}$, and let $\Delta_{n-1}=|A|$, the $(n-1)$-simplex spanned by the vectors 
$e(A_1),\dots,e(A_n)$ in $C_b(X_n)$.
Clearly, $C_b(X_n)$ is nothing but $\R^n$ with the $l_\infty$ norm $||(x_1,\dots,x_n)||=\max(x_1,\dots,x_n)$, 
and its points $e(A_1),\dots,e(A_n)$ are of the form $(0,1,\dots,1),\,(0,0,1,\dots,1),\dots,(0,\dots,0)$.
Hence $\Delta_{n-1}$ is the standard skew $(n-1)$-simplex $\{(0,x_2,\dots,x_n)\mid 0\le x_2\le\dots\le x_n\le 1\}$ 
with the $l_\infty$ metric.

(b) Let $X$ be any metric space with $d(x,y)=1$ for $x\ne y$ and let $B\:[n]\to K(X)$ be any injective chain.
Let $X_n$ be as in (a) and let $g\:X_n\to X$ be an embedding such that $g(a_1)\in B_1$ and each 
$g(a_{i+1})\in B_{i+1}\but B_i$.
Then $g$ induces the restriction map $g^*\:C_b(X)\to C_b(X_n)$, which clearly restricts to an isometry between 
$|B|$ and $\Delta_{n-1}$.
\end{example}

\begin{example} (a) Let $X$ be a metric space consisting of $3$ points: $a,b,c$.
Let $A_1=\{a\}$, $A_2=\{a,b\}$ and $A_3=\{a,b,c\}$.
Let $p=d(a,b)$, $q=d(b,c)$ and $r=d(a,c)$.
We have $P\bydef d(A_1,A_2)=p$, $Q\bydef d(A_2,A_3)=\min(q,r)$ and $R\bydef d(A_1,A_3)=\max(p,r)$.
Let us note that $P,Q\le R\le P+Q$.
Then $|A|$ is the $2$-simplex spanned by $e(A_1)=(0,p,r)$, $e(A_2)=(0,0,Q)$ and $e(A_3)=(0,0,0)$ in $C_b(X)=\R^3$ 
with the $l_\infty$ metric.
The edges of $|A|$ are of lengths $Q$, $\max(p,r)=R$ and $\max(p,r-Q)=\max(p,r-q,r-r)=P$ (using that $p+q\ge r$).

Let us note that when $R>P$, we have $r=R$ and consequently $P$, $Q$, $R$ determine the vertices 
$(0,P,R)$, $(0,0,Q)$ and $(0,0,0)$ of $|A|$.
Thus when $R>P$, the metric on $|A|$ is determined by the edge lengths (i.e., by its restriction to the set 
of vertices).

(b) Let $Y$ be a metric space consisting of $4$ points: $a,b,c_+,c_-$.
Let $B_1=\{a\}$, $B_2=\{a,b\}$ and $B_3=\{a,b,c_+,c_-\}$.
Let $p=d(a,b)$, $q_\pm=d(b,c_\pm)$ and $r_\pm=d(a,c_\pm)$.
We have $P\bydef d(B_1,B_2)=p$, $Q\bydef d(B_2,B_3)=\max(Q_+,Q_-)$, where $Q_\pm=\min(q_\pm,r_\pm)$, and 
$R\bydef d(B_1,B_3)=\max(p,r_+,r_-)$.
Then $|B|$ is the $2$-simplex spanned by $e(B_1)=(0,p,r_+,r_-)$, $e(B_2)=(0,0,Q_+,Q_-)$ and 
$e(B_3)=(0,0,0,0)$ in $C_b(Y)=\R^4$ with the $l_\infty$ metric.
Then the edges of $|B|$ are of lengths $\max(Q_+,Q_-)=Q$, $\max(p,r_+,r_-)=R$ and $\max(p,r_+-Q_+,r_--Q_-)=P$.

Let us note that the metric of $|B|$ is not determined by the edge lengths even when $R>P$, and consequently
$|B|$ is generally not isometric to any of the simplexes $|A|$ described in (a).
Indeed, the distance from the vertex of $|B|$ at the origin to the middle of the opposite side
equals $L\bydef \max\big(p,\frac{r_++Q_+}2,\frac{r_-+Q_-}2\big)$.
Let us assume for simplicity that $p<q_\pm<r_\pm$, so that $Q_\pm=q_\pm$.
Then $Q=\max(q_+,q_-)$ and $R=\max(r_+,r_-)$ do not determine $L=\max\big(\frac{r_++q_+}2,\frac{r_-+q_-}2\big)$.
For instance, if $q_+=q_--\eps$ and $r_+=r_-+\eps$ for some $\eps>0$, then 
$L=(q_-+r_+-\eps)/2=(Q+R-\eps)/2$, where $\eps$ can vary.

(c) Let $Z$ be any metric space and let $C\:[n]\to K(Z)$ be any injective chain.
Since the $C_i$ are compact, for each $i<j$, $d(C_i,C_j)=d(C_i,x_{ij})$ for some $x_{ij}\in C_j$.
Let $Z'$ be the finite subspace of $Z$ consisting of the $x_{ij}$ for all $i<j$.
Then the restriction map $r\:C_b(Z)\to C_b(Z')$ restricts to an isometry on the vertices of $|C|$.
However, $r$ need not restrict to an isometry on $|C|$, because the metric on $r(|C|)$ generally
depends on $d(x_{ij},C_k)$, which depend on the choice of the $x_{ij}$.
Indeed, for $n=2$ the metric on $r(|C|)\subset\{0\}\x\R^3\subset\R^4$ is as described in (b), and we have
seen that it does depend on the additional parameters.
\end{example}

\subsection{Simplicial hyperspace of a metrizable space}

If $X$ is a metric space of diameter $\le 1$, let $K_\Delta(X)$ denote the union 
$\bigcup_{C\in K(X)^{[n]},\,n\in\N}|C|$ of the convex hulls in $C_b(X)$ of all nonempty finite chains in $K(X)$.
When all distances in $X$ are equal to $1$, $K_\Delta(X)$ is isometric to the geometric realization of
$K(X)$ as a (discrete) poset \cite{M3} (see Example \ref{discrete-hull}). 
In general, let us note that although the topology of $C_b(X)$ does not depend on the metric of $X$, 
the subset $K_\Delta(X)$ of $C_b(X)$ depends on the embedding $e\:K(X)\to C_b(X)$, which in turn depends 
on the metric of $X$.

Let $X_+=X\sqcup\{p\}$, where $d(p,x)=1$ for each $x\in X$.
Here $X$ retains the original metric of diameter $\le 1$, so the inclusion $K(X)\subset K(X_+)$ is an isometry.
Let $K^+_\Delta(X)$ be the union of the convex hulls in $C_b(X_+)$ of all nonempty finite chains in 
$K(X)\subset K(X_+)$.

\begin{lemma} \label{+isometry} $K^+_\Delta(X)$ is isometric to $K_\Delta(X)$.
\end{lemma}

\begin{proof}
Let us show that the restriction map $r\:C_b(X_+)\to C_b(X)$ restricts to an isometry between 
$K^+_\Delta(X)$ and $K_\Delta(X)$.
Indeed, for any $F,G\in K_\Delta(X)$ we have $F(x)=\sum_{i=1}^n\lambda_i d(x,A_i)$
and $G(x)=\sum_{j=1}^m\mu_j d(x,B_j)$ for some finite chains $A\:[n]\to K(X)$ and $B\:[m]\to K(X)$,
where each $\lambda_i>0$, each $\mu_j>0$ and $\sum_{i=1}^n\lambda_i=1=\sum_{j=1}^m\mu_j$.
Then $F(p)=1=G(p)$, and consequently $\sup_{x\in X_+}|F(x)-G(x)|=\sup_{x\in X}|F(x)-G(x)|$.
\end{proof}

\begin{lemma} \label{estimates}
For each $n\in\N$, each $\Gamma\in(0,1]$ and each $\eps>0$ there exists a $\delta>0$ such that
the following holds.

Let $A\:[n]\to K(X)\subset K(X_+)$ and $B\:[m]\to K(X)\subset K(X_+)$ be injective monotone maps such that 
each $d(A_i,A_{i+1})\ge\Gamma$ and each $d(B_j,B_{j+1})\ge\Gamma$.
Let $f_i=e(A_i)$ and $g_j=e(B_j)$.
Let $F=\sum_{i=1}^n\lambda_i f_i$ and $G=\sum_{j=1}^m\mu_j g_j$, where $1\ge\lambda_i\ge\Gamma$ and 
$1\ge\mu_j\ge\Gamma$ for each $i$ and $j$.

If $||F-G||\le\delta$, then $m=n$, each $d(A_i,B_i)\le\eps$ and each $|\lambda_i-\mu_i|\le\eps$. 
\end{lemma}

Lemma \ref{estimates} is not used in the sequel.
However, it is a simplified version of Lemma \ref{estimates2}, whose proof might be easier
to read after that of Lemma \ref{estimates}.

The proof of Lemma \ref{estimates} is in turn an elaboration on that of Lemma \ref{simplicial hyperspace-def}(a).

\begin{proof} We may assume that $\eps\le 1$ and $\eps\le\Gamma$ (by decreasing $\eps$ if needed).
Let $\delta=(\eps/2)^{2n}$.

If $n=0$ but $m>0$, then $||G||\ge\mu_1||g_1||\ge\Gamma$ since $||g_1||=1$. 
Hence $||F||\ge\Gamma-\delta>0$, which is a contradiction.
This establishes the assertion for $n=0$.
Also, if $n>0$, a similar argument shows that $m>0$.  
Arguing by induction, we may assume that the assertion is known if $n$ is replaced by a smaller number.

If $a\in A_1$, then $a$ also lies in each $A_i$, and hence $F(a)=0$.
On the other hand, since each $g_i(a)\ge 0$ and each $\mu_i\ge 0$, we have 
$\mu_1g_1(a)\le G(a)\le F(a)+\delta=\delta$.
Hence $d(a,B_1)=g_1(a)\le\delta/\mu_1\le\delta/\Gamma$.
Similarly, if $b\in B_1$, then $d(A_1,b)\le\delta/\Gamma$.
Hence $d(A_1,B_1)\le\delta/\Gamma$.
Therefore also $||f_1-g_1||\le\delta/\Gamma$.

For each $x\in A_2$ (or for each $x\in X_+$ if $n=1$), $\lambda_1f_1(x)=F(x)\ge G(x)-\delta\ge\mu_1g_1(x)-\delta\ge
\mu_1f_1(x)-\mu_1\delta/\Gamma-\delta$.
Since $\mu_1\le 1$ and $\Gamma\le 1$, we have $\mu_1f_1(x)-\lambda_1f_1(x)\le 2\delta/\Gamma$.
Since $d(A_1,A_2)\ge\Gamma$, there exists an $a\in A_2$ such that $f_1(a)=d(a,A_1)\ge\Gamma$.
(If $n=1$, let $a=p$, the point in $X_+\but X$.)
Then $\mu_1-\lambda_1\le 2\delta/\Gamma f_1(a)\le 2\delta/\Gamma^2$.
Similarly, $\lambda_1-\mu_1\le 2\delta/\Gamma^2$, so $|\lambda_1-\mu_1|\le 2\delta/\Gamma^2$.
Then $||\lambda_1f_1-\mu_1g_1||\le||\lambda_1(f_1-g_1)||+||(\lambda_1-\mu_1)g_1||\le
\delta/\Gamma+2\delta/\Gamma^2\le 3\delta/\Gamma^2$ using that $\lambda_1\le 1$, $||g_1||\le 1$ and $\Gamma\le 1$.

We have $||\sum_{i=2}^n\lambda_i f_i-\sum_{j=2}^m\mu_j g_j||=||F-G+\mu_1g_1-\lambda_1f_1||\le
||F-G||+||\mu_1g_1-\lambda_1f_1||\le\delta+3\delta/\Gamma^2\le 4\delta/\Gamma^2$.
We have $4\delta/\Gamma^2=4(\eps/2)^{2n}/\Gamma^2\le (\eps/2)^{n-1}$.
Then by the induction hypothesis, $m=n$, $d(A_i,B_i)\le\eps$ for each $i\ge 2$ and $|\lambda_i-\mu_i|\le\eps$ 
for each $i\ge 2$.
\end{proof}

\begin{lemma} \label{estimates2}
For each $n\in\N$, each $\Gamma\in(0,1]$ and each $\eps>0$ there exists a $\delta>0$ such that
the following holds.

Let $A\:[n]\to K(X)\subset K(X_+)$ and $B\:[m]\to K(X)\subset K(X_+)$ be injective monotone maps such that 
each $d(A_i,A_{i+1})\ge\Gamma$.
Let $f_i=e(A_i)$ and $g_j=e(B_j)$.
Let $F=\sum_{i=1}^n\lambda_i f_i$ and $G=\sum_{j=1}^m\mu_j g_j$, where $1\ge\lambda_i\ge\Gamma$ for each $i$
and $1\ge\mu_j\ge 0$ for each $j$.

If $||F-G||\le\delta$, then there exist $1=l_0\le k_1\le l_1\le k_2\le\dots\le l_n\le k_{n+1}=m$ 
such that $d(A_i,B_j)\le\eps$ whenever $k_i+1\le j\le l_i$, each $|\lambda_i-\sum_{j=k_i+1}^{l_i}\mu_j|\le\eps$, 
and each $\sum_{j=l_i+1}^{k_i}\mu_j\le\eps$. 
\end{lemma}

\begin{proof}
We may assume that $\eps\le 1$ and $\eps\le\Gamma$ (by decreasing $\eps$ if needed).
Let $\beta=(\eps/6)^{\phi(n)}$, where $\phi(0)=1$
and $\phi(n)=4\phi(n-1)+1$.
Let $\alpha=\beta^2$ and $\delta=\alpha^2$.
Thus $\delta=(\eps/6)^{4\phi(n)}$.

If $n=0$, let us consider the point $p$ in $X_+\but X$.
Since $F(p)=0$, we have $\sum_{i=1}^m\mu_i=\sum_{i=1}^m\mu_ig_i(p)=G(p)\le\delta=\eps/6$.
Arguing by induction, we may assume that the assertion is known if $n$ is replaced by a smaller number.

If $n>0$, then $||F||\ge\lambda_1||f_1||\ge\Gamma$ since $||f_1||=1$. 
Hence $||G||\ge\Gamma-\delta>0$, and so $m>0$.

Let $I=\{i\in[m]\mid\exists a_i\in A_1\text{ such that }d(a_i,B_i)\ge\alpha\}$.
If $I\ne\emptyset$ and $k$ is the greatest element of $I$, then there exists an $a\in A_1$ 
such that $d(a,B_k)\ge\alpha$.
Since $B_1\subset\dots\subset B_k$, we also have $d(a,B_i)\ge\alpha$ for all $i\le k$.
Hence $I=\{1,\dots,k\}$ and the same $a_i=a$ works for each $i\in I$.
If $I=\emptyset$, we let $k=0$.

Let $J=\{j\in[m]\mid\exists b_j\in B_j\text{ such that }d(A_1,b_j)\ge\beta\}$.
If $J\ne\emptyset$ and $l+1$ is the least element of $J$, then there exists a $b\in B_{l+1}$
such that $d(A_1,b)\ge\beta$.
Since $B_{l+1}\subset\dots\subset B_m$, we also have $b\in B_j$ for all $j\ge l+1$.
Hence $J=\{l+1,\dots,m\}$ and the same $b_j=b$ works for each $j\in J$.
If $J=\emptyset$, we let $l=m$ and $b=p$ (the point in $X_+\but X$).

Let us note that if $i\ge k+1$, then $d(x,B_i)<\alpha$ for each $x\in A_1$, so by Lemma \ref{Hausdorff}(a),
$g_i(x)-f_1(x)=d(x,B_i)-d(x,A_1)\le\alpha$ for each $x\in X$.
Similarly, if $i\le l$, then $d(x,A_1)<\beta$ for each $x\in B_i$, and so by Lemma \ref{Hausdorff}(a),
$f_1(x)-g_i(x)=d(x,A_1)-d(x,B_i)\le\beta$ for each $x\in X$.
In particular, if $k+1\le i\le l$, then $d(A_1,B_i)\le\max(\alpha,\beta)=\beta<\eps$. 

Let $\kappa=\sum_{i=1}^k\mu_i$.
Since $a\in A_1\subset\dots\subset A_n$, we have $F(a)=0$, so $G(a)\le\delta$.
Since $g_i(a)=d(a,B_i)\ge\alpha$ for each $i\le k$, and each $\mu_ig_i\ge 0$, we have 
$\kappa\alpha\le\sum_{i=1}^k\mu_ig_i(a)\le G(a)\le\delta$.
Hence $\kappa\le\delta/\alpha=\alpha$.

For each $x\in B_{l+1}$ (or for each $x\in X_+$ if $l=m$) we have $G(x)=\sum_{i=1}^l\mu_ig_i(x)$.
Since each $\lambda_if_i\ge 0$, we get 
\[\lambda_1f_1(x)\le F(x)\le G(x)+\delta=\sum_{i=1}^l\mu_ig_i(x)+\delta.\tag{$*$}\]
Since $f_1(b)=d(b,A_1)\ge\beta$, we further get $\sum_{i=1}^l\mu_ig_i(b)\ge\lambda_1\beta-\delta$.
On the other hand, since each $g_i(b)\le 1$, we have $\sum_{i=1}^k\mu_ig_i(b)\le\kappa\le\alpha$.
Since $\alpha+\delta\le 2\alpha=2\beta^2<\Gamma\beta\le\lambda_1\beta$, 
we have $\alpha<\lambda_1\beta-\delta$, and we conclude that $l>k$ (in particular, $l\ge k$ and
$m>0$).

Let $\nu_1=\sum_{i=k+1}^l\mu_i$.
Since $g_i(x)\le f_1(x)+\alpha$ for each $x\in X$ and each $i\ge k+1$, from ($*$) we have 
$\lambda_1f_1(x)\le\sum_{i=1}^l\mu_ig_i(x)+\delta\le\kappa+
\sum_{i=k+1}^l\mu_i\big(f_1(x)+\alpha\big)+\delta\le
\alpha+\nu_1f_1(x)+\nu_1\alpha+\delta$
for each $x\in B_{l+1}$ (or for each $x\in X_+$ if $l=m$).
Since $\nu_1\le 1$, we have $\lambda_1f_1(x)-\nu_1f_1(x)\le\alpha+\nu_1\alpha+\delta\le 3\alpha$.
Since $f_1(b)\ge\beta$, we get $\lambda_1-\nu_1\le 3\alpha/f_1(b)\le 3\alpha/\beta=3\beta$.

For each $x\in A_2$ (or for each $x\in X_+$ if $n=1$) we have $F(x)=\lambda_1f_1$.
Since $g_i(x)\ge f_1(x)-\beta$ for each $i\le l$, we have
$\lambda_1f_1(x)=F(x)\ge G(x)-\delta\ge\sum_{i=k+1}^l\mu_ig_i(x)-\delta\ge
\sum_{i=k+1}^l\mu_i\big(f_1(x)-\beta\big)-\delta=\nu_1f_1(x)-\nu_1\beta-\delta$.
Since $\nu_1\le 1$, we have
$\nu_1f_1(x)-\lambda_1f_1(x)\le\nu_1\beta+\delta\le 2\beta$.
Since $d(A_1,A_2)\ge\Gamma$, there exists an $a'\in A_2$ such that $f_1(a')=d(a',A_1)\ge\Gamma$.
(If $n=1$, let $a'=p$, the point in $X_+\but X$.)
Then $\nu_1-\lambda_1\le 2\beta/f_1(a')\le 2\beta/\Gamma$.

Thus we get $|\lambda_1-\nu_1|\le 3\beta/\Gamma$.
Then $||\lambda_1f_1-\sum_{i=k+1}^l\mu_ig_i||\le
||(\lambda_1-\nu_1)f_1||+||\sum_{i=k+1}^l\mu_i(f_1-g_i)||\le 3\beta/\Gamma+\nu_1\beta\le 4\beta/\Gamma$
using that $||f_1||\le 1$, each $||f_1-g_i||\le\beta$ and $\nu_1\le 1$.
Also $||\sum_{i=1}^k\mu_ig_i||\le\kappa\le\alpha\le\beta/\Gamma$ using that each $||g_i||\le 1$.
Hence $||\lambda_1f_1-\sum_{i=1}^l\mu_ig_i||\le 5\beta/\Gamma$.

We have $||\sum_{i=2}^n\lambda_i f_i-\sum_{j=l+1}^m\mu_j g_j||=||F-G+\sum_{i=1}^l\mu_ig_i-\lambda_1f_1||\le
||F-G||+||\sum_{i=1}^l\mu_ig_i-\lambda_1f_1||\le\delta+5\beta/\Gamma\le 6\beta/\Gamma$.
We have $6\beta/\Gamma\le 6(\eps/6)^{\phi(n)}/\eps=(\eps/6)^{4\phi(n-1)}$.
The assertion now follows from the induction hypothesis.
\end{proof}

\begin{theorem} \label{simplicial-hyperspace} 
The topology of $K_\Delta(X)$ depends only on the topology of $X$.
\end{theorem}

Unfortunately, the uniform structure of $K_\Delta(X)$ does not seem to depend only on the uniform structure 
of $X$ (at least, the estimates obtained below depend on the dimensions of the two convex hulls).

\begin{proof}
By Lemma \ref{+isometry} it suffices to show that the topology of $K^+_\Delta(X)$ depends only on the topology of $X$.
This makes some difference since $||F||=1$ for each $F\in K^+_\Delta(X)$, but e.g.\ $||e(X)||=0$ where 
$e(X)\in K_\Delta(X)$.

It is well-known that the topology induced on $K(X)$ by the Hausdorff metric is the Vietoris topology,
which is independent of the metric \cite{Mich}*{Theorem 3.3}.
So if $S$ denotes the underlying set of the metric space $X=(S,d)$, and $Y=(S,d')$ for some metric $d'$ on $S$
inducing the same topology, then $\id_S\:K(X)\to K(Y)$ is a homeomorphism.
Let us extend the composition 
$t\:e_{X^+}\big(K(X)\big)\xr{e_{X^+}^{-1}}K(X)\xr{\id_S}K(Y)\xr{e_{Y^+}}e_{Y^+}\big(K(Y)\big)$
linearly to a map $T\:K_\Delta^+(X)\to K_\Delta^+(Y)$.
By Lemma \ref{simplicial hyperspace-def}, $T$ is a bijection.
We will show that $T$ is continuous; by symmetry, $T^{-1}$ will then also be continuous.

Given an $x\in K_\Delta^+(X)$, let $|A|$ be the minimal simplex of $K_\Delta^+(X)$ containing $x$; thus 
$A\:[n]\to K(X)$ is a nonempty finite chain of nonempty subsets $A_1\subset\dots\subset A_n\subset X$ 
for some $n=n(x)$.
We have $x=\sum_{i=1}^n\lambda_if_i$, where each $f_i=e(A_i)$ and each $\lambda_i\ge 0$, with $\sum\lambda_i=1$.
By the minimality, $A$ is injective, i.e., each $A_{i+1}\ne A_i$, and also each $\lambda_i>0$.
Then there exists a $\Gamma\in(0,1]$ such that each $d(A_i,A_{i+1})\ge\Gamma$ and each $\lambda_i\ge\Gamma$.

Let $\eps>0$ be given.
We need to show that there exists a $\delta>0$ such that if $y\in K_\Delta^+(X)$ is $\delta$-close to $x$,
then $T(y)$ is $\eps$-close to $T(x)$.
Since $t$ is continuous, there exists an $\alpha>0$ such that $t$ sends the $\alpha$-neighborhood
of each $f_i$ into the $\beta$-neighborhood of $t(f_i)$, where $\beta=\eps/(3n+1)$.
We may assume that $\alpha\le\beta$.
Let $\delta=\delta_{\ref{estimates2}}$ be given by Lemma \ref{estimates2} for 
$\eps_{\ref{estimates}}=\alpha$ and $\Gamma_{\ref{estimates}}=\Gamma$.

Suppose that $y\in K_\Delta^+(X)$ is $\delta$-close to $x$.
Let $|B|$ be the minimal simplex of $K_\Delta^+(X)$ containing $y$; thus $B\:[m]\to K(X)$ is injective, for
some $m=m(y)$, and $y=\sum_{i=1}^m\mu_ig_i$, where each $g_i=e(B_i)$ and each $\mu_i>0$, with 
$\sum_{i=1}^m\mu_i=1$.
Then by Lemma \ref{estimates2} there exist $1=l_0\le k_1\le l_1\le k_2\le\dots\le l_n\le k_{n+1}=m$ 
such that $||f_i-g_j||\le\alpha$ whenever $k_i+1\le j\le l_i$, 
each $|\lambda_i-\sum_{j=k_i+1}^{l_i}\mu_j|\le\alpha\le\beta$, 
and each $\sum_{j=l_i+1}^{k_i}\mu_j\le\alpha\le\beta$.
We have $T(x)=\sum_{i=1}^n\lambda_if_i'$ and $T(y)=\sum_{j=1}^m\mu_jg'_j$, where $f'_i=t(f_i)$ and $g'_j=t(g'_j)$.
By the above, $||f'_i-g'_j||\le\beta$ whenever $k_i+1\le j\le l_i$.
For each $i$, since $||f'_i||\le 1$ and $\sum_{j=k_i+1}^{l_i}\mu_j\le 1$, we have 
$||\lambda_if'_i-\sum_{j=k_i+1}^{l_i}\mu_jg'_j||\le
||\big(\lambda_i-\sum_{j=k_i+1}^{l_i}\mu_j\big)f'_i||+||\sum_{j=k_i+1}^{l_i}\mu_j(f'_i-g'_j)||\le
\beta+\beta$.
Since each $||g'_j||\le 1$, we also have $||\sum_{j=l_i+1}^{k_i}\mu_jg'_j||\le\beta$.
Hence $||T(x)-T(y)||=||\sum_{i=1}^n\lambda_if'_i-\sum_{j=1}^m\mu_jg'_j||\le(3n+1)\beta=\eps$.
\end{proof}

\subsection{The order complex of the hyperspace}

\begin{theorem} \label{topology3}
$K_\Delta(X)$ is homeomorphic to $|K(X)|$.
\end{theorem}

\begin{proof} By Lemma \ref{+isometry}, it suffices to show that $K_\Delta^+(X)$ is homeomorphic to $|K(X)|$.

Given an $x\in K_\Delta^+(X)$ and an $y\in K_\Delta^+(X)$, let $|A|$ and $|B|$ be the minimal simplexes 
of $K_\Delta^+(X)$ containing $x$ and $y$.
Thus $A\:[n]\to K(X)$ and $B\:[m]\to K(X)$ are nonempty finite chains of nonempty subsets 
$A_1\subset\dots\subset A_n\subset X$ and $B_1\subset\dots\subset B_m\subset X$ for some $n$ and $m$,
which are injective by the minimality.
We have $x=\sum_{i=1}^n\lambda_if_i$ and $y=\sum_{i=1}^m\mu_ig_i$, where each $f_i=e(A_i)$, each $g_i=e(B_i)$, 
each $\lambda_i>0$ and each $\mu_i>0$, with $\sum\lambda_i=\sum\mu_i=1$.

We also have the step functions $\phi_x\bydef \phi_{\sum\lambda_i A_i},\,\phi_y\bydef \phi_{\sum\mu_i B_i}\:I\to K(X)$.
Let us prove that the map $\Phi\:K_\Delta^+(X)\to|K(X)|$, defined by $x\mapsto\phi_x$, is continuous.
Let us fix $x$.
Then there exists a $\Gamma\in(0,1]$ such that each $d(A_i,A_{i+1})\ge\Gamma$ and each $\lambda_i\ge\Gamma$.
Let $\eps>0$ and let $\delta=\delta(\eps,n,\Gamma)$ be given by Lemma \ref{estimates2}.
If $y$ is $\delta$-close to $x$, then by Lemma \ref{estimates2} there exist 
$1=l_0\le k_1\le l_1\le k_2\le\dots\le l_n\le k_{n+1}=m$ such that $d(A_i,B_j)\le\eps$ whenever 
$k_i+1\le j\le l_i$, each $|\lambda_i-\sum_{j=k_i+1}^{l_i}\mu_j|\le\eps$, 
and each $\sum_{j=l_i+1}^{k_i}\mu_j\le\eps$.

Using the notation of \cite{M2}*{\S\ref{metr:memefu}},
we have $|\phi_x\nee\eps\phi_y|\subset\bigcup_{i=0}^n [\min(p_i,r_i),\,\max(p_i,q_{i+1})]$, where
$p_i=\sum_{j\le i}\lambda_i$, $q_i=\sum_{j\le k_i}\mu_j$ and $r_i=\sum_{j\le l_i}\mu_j$.
Then $|p_i-r_i|\le 2i\eps\le 2n\eps$ and $|p_i-q_{i+1}|\le (2i+1)\eps\le (2n+1)\eps$ for each $0\le i\le n$.
Hence $\max(p_i,q_{i+1})-\min(p_i,r_i)\le (4n+1)\eps$, and consequently
$\mu\big(|\phi_x\nee\eps\phi_y|\big)\le (4n+1)(n+1)\eps$.
Thus $D(\phi_x,\phi_y)\le\eps+\mu\big(|\phi_x\nee\eps\phi_y|\big)\le \big((4n+1)(n+1)+1\big)\eps$.

It remains to show that $\Phi^{-1}$ is continuous; we will show that it is in fact uniformly continuous.
Indeed, by considering a common subdivision of the two triangulations of $[0,1]$ (one with vertices 
$v_i\bydef \sum_{j\le i}\lambda_i$ and another with vertices $w_i\bydef \sum_{j\le i}\mu_i$), we may assume that 
$m=n$ and each $\lambda_i=\mu_i$.
Then $||x-y||\le\sum_{i=1}^n\lambda_i||f_i-g_i||=\sum_{i=1}^n\lambda_i d(A_i,B_i)=L_1(\phi_x,\phi_y)$.
\end{proof}

\section{Derived limits over a topological poset} \label{mainsection}

\subsection{Partially ordered sheaves}\label{posheaves}

A {\it diagram of spaces} indexed by a topological poset $P$ is a morphism of topological posets
$f\:Q\to P$ (i.e., a continuous monotone map) such that if $f(q)\le p$, then there exists a unique 
$q'\in f^{-1}(p)$ such that $q\le q'$.%
\footnote{In Logic, monotone maps (between Kripke frames or Heyting algebras) satisfying this property 
are known as ``$p$-morphisms''.
In Category Theory, such maps between posets (viewed as categories) are known as ``discrete opfibrations'', 
and are a special case of ``Grothendieck opfibrations''.}
It is not hard to see (cf.\ \cite{M4}*{\S\ref{comb:hatcher-maps}, \S\ref{comb:closed-open-maps}}) that this 
condition is equivalent to saying that $f$ is 
an order-closed map with order-discrete point-inverses (with respect to the Alexandrov topologies on 
$P$ and $Q$).
Note also that the monotonicity of $f$ is equivalent to its order-continuity.
If $f\:Q\to P$ is a diagram of spaces, we have the Hatcher maps $f^p_{p'}\:f_p\to f_{p'}$,
where $f_p=f^{-1}(p)$, defined whenever $p\le p'$, with $f^p_p=\id_{f_p}$ (compare 
\cite{M4}*{\S\ref{comb:hatcher-maps}}).

The {\it homotopy colimit} of the diagram of spaces $f$ is the natural continuous map $\Hocolim f\:|Q|\to |P|$ 
(compare \cite{M4}*{\S\ref{comb:homotopy-colimit}}).

\begin{example} \label{atoms-diagram}
Let $P$ be a topological poset and $A(P)$ its set of atoms, and let $E(P)$ be the topological 
subposet of $A(P)\x P$ consisting of all pairs $(a,p)$ such that $a\le p$ (see details in \S\ref{resolution}).
Then the projection $E(P)\subset A(P)\x P\to P$ is a diagram of spaces.
\end{example}

Let $\C$ be a concrete category over the category of sets (for example, the category of abelian groups). 
A $\C$-valued {\it sheaf-diagram} over a topological poset $P$ is a $\C$-valued sheaf
$\F\:Q\to P$ that is also a diagram of spaces whose Hatcher maps are $\C$-morphisms.
In other words, $\F$ is a $\C$-valued sheaf over the space $P$ along with $\C$-morphisms 
$\F^p_q\:\F_p\to\F_q$ for all pairs $p\le q$ such that each $\F^p_p=\id_P$ and each 
$\F^p_q\F^q_r=\F^p_r$.

A $\C$-valued {\it posheaf} (=partially ordered sheaf) is a $\C$-valued sheaf-diagram $\F\:Q\to P$ 
such that the restriction $\F_\le\:\le_Q\to\le_P$ of $\F\x\F\:Q\x Q\to P\x P$ is a sheaf 
($\C$-valued or equivalently set-valued).
Let us note that since $\F$ is a sheaf, so is $\F\x\F$ and consequently also its restriction over 
any subset of $P\x P$, in particular, over $\le_P$.
Thus $\F_\le$ is a posheaf if and only if $\le_Q$ is open in $(\F\x\F)^{-1}(\le_P)$.

\begin{lemma} \label{posheaf}
Let $\F\:Q\to P$ be a posheaf of sets.

(a) Each $q\in Q$ has a neighborhood $U$ such that $\F(U)$ is a neighborhood of $\F(q)$ in $P$ and $\F$
restricts to a homeomorphism $U\to\F(U)$ that is also an isomorphism of posets.

(b) Let $\pi_1,\pi_2\:\le_P\subset P\x P\to P$ be the projections onto the two factors.
Then the map $f\:\pi_1^*Q\to \pi_2^*Q$ given by $\F^p_{p'}\:\F_p\to\F_{p'}$ over each pair 
$(p\le p')\in\le_P$ is continuous, and hence is a morphism of sheaves $\pi_1^*\F\to\pi_2^*\F$.
\end{lemma}

\begin{proof}[Proof. (a)] Since $\F$ is a sheaf, there exists an open neighborhood $V$ of $q$ in $Q$ 
such that $\F(V)$ is a neighborhood of $\F(q)$ in $P$ and $\F$ restricts to a homeomorphism $V\to\F(V)$.
Since $\F_\le$ is also a sheaf, there exists an open neighborhood $W$ of $(q,q)$ in $\le_Q$ such that 
$\F_\le(W)$ is an open neighborhood of $\F_\le(q,q)$ in $\le_P$ and $\F_\le$ restricts to a 
homeomorphism $W\to\F(W)$.
By the definition of product topology, $W$ contains $W_1\x W_2\cap\le_Q$, where $W_1$ and $W_2$
are open neighborhoods of $q$ in $Q$.
Let $O=V\cap W_1\cap W_2$ and $\le_O=O\x O\cap\le_Q$.
Then $\F(O)$ is a neighborhood of $\F(q)$ in $P$ and $\F$ restricts to a homeomorphism $O\to\F(O)$;
also, $\F_\le(\le_O)$ is an open neighborhood of $\F_\le(q,q)$ in $\le_P$ and $\F_\le$ restricts to 
a homeomorphism $\le_O\to\F_\le(\le_O)$.
By the definition of product topology, $\F_\le(\le_O)$ contains $O_1\x O_2\cap\le_P$, where $O_1$ 
and $O_2$ are open neighborhoods of $\F(q)$ in $P$.
Let $U=O\cap\F^{-1}(O_1\cap O_2)$ and $\le_U=U\x U\cap\le_U$.
Then $\F(U)$ is a neighborhood of $\F(q)$ in $P$ and $\F$ restricts to a homeomorphism $U\to\F(U)$;
also, $\F_\le(\le_U)=\F(U)\x\F(U)\cap\le_P$.
If $p,p'\in\F(U)$ and $p\le p'$, there exist unique $q,q'\in U$ such that $\F(q)=p$ and $\F(q')=p'$,
and since $\F\x\F$ restricts to a homeomorphism between $U\x U\but\le_U$ and $\F(U)\x\F(U)\but\le_P$,
we must have $q\le q'$.
Thus $\F|_U\:U\to\F(U)$ is also an isomorphism of posets.
\end{proof}

\begin{proof}[(b)] Let us note that $\pi_1^*Q$ consists of pairs $(q,p)$, where $q\in Q$, $p\in P$
and $\F(q)\le p$.
Similarly, $\pi_2^*Q$ consists of pairs $(p,q)$, where $q\in Q$, $p\in P$ and $p\le\F(q)$.
Clearly, $\F_\le:\le_Q\to\le_P$ factors through $\pi_i^*Q$ for each $i$.
Since $\F_\le$ is a sheaf and $\pi_i^*\F\:\pi_i^*Q\to\le_P$ is a sheaf, the resulting map
$\rho_i\:\le_Q\to\pi_i^*Q$ is also a sheaf.
In particular, it is open and continuous.
On the other hand, $\rho_1$ is a bijection.
Indeed, $\rho_1(q,q')=(q,p')$, where $p'=\F(q')$, and we have $q'=\F^p_{p'}(q)$, where $p=\F(q)$.
Thus $\rho_1$ is a homeomorphism, and we may define $f$ to be the composition $\rho_2\rho_1^{-1}$.
\end{proof}

We recall that the Leray sheaf $\h^n(\pi)$ of a continuous map $\pi\:E\to B$ is the sheafafication of the presheaf 
$U\mapsto H^n\big(\pi^{-1}(U)\big)$.

\begin{lemma} \label{closed Leray}
If $\pi\:E\to B$ is a closed map, where $E$ is metrizable, then
$\h^n(\pi)_b\simeq H^n\big(\pi^{-1}(b)\big)$ for each $b\in B$.
\end{lemma}

This result is well-known (see \cite{Br}*{Proposition IV.4.2 and Remark 2 to Theorem II.10.6}).
We include a self-contained proof for convenience.

\begin{proof} Since $\pi$ is closed, every open neighborhood $V$ of $\pi^{-1}(b)$ in $E$ contains 
one of the form $\pi^{-1}(U)$, where $U$ is an open neighborhood of $b$ in $B$; 
namely, $U=B\but\pi(E\but V)$.
Hence the group $\h^n(\pi)_b=\colim_U H^n\big(\pi^{-1}(U)\big)$, where $U$ runs over all open neighborhoods of $b$ 
in $B$, is isomorphic to $\colim_V H^n(V)$, where $V$ runs over all open neighborhoods of $\pi^{-1}(b)$ in $E$.
By Spanier's tautness theorem (see \cite{Sp}*{Theorem 6.6.3}), the latter group is isomorphic to
$H^n\big(\pi^{-1}(b)\big)$.
\end{proof}

\begin{theorem} \label{hyperspace-posheaf}
Let $X$ be a metrizable space and let $E(X)$ be the subspace of $X\x K(X)$ consisting of all pairs 
$(x,K)$ such that $x\in K$.
Let $\pi$ be the composition of the inclusion $E(X)\subset X\x K(X)$ and the projection $X\x K(X)\to K(X)$.
Then

(a) $\pi$ is a closed map, and $\h^n(\pi)_A\simeq H^n(A)$ for each compact $A\subset X$;

(b) $\h^n(\pi)$ is a posheaf with respect to Hatcher maps $\h^n(\pi)^B_A\:\h^n(\pi)_B\to\h^n(\pi)_A$, 
$A\subset B$, defined as the restriction (=inclusion induced) maps $H^n(B)\to H^n(A)$.
\end{theorem}

Let us note that $K(X)$ is ordered by {\it reverse} inclusion in (b).

\begin{proof}[Proof. (a)] Suppose that $F\subset E(X)$ is a closed subset such that $\pi(F)$ is not closed.
Then there exists a sequence of points $A_n\in\pi(F)$ converging to a point $A\notin\pi(F)$.
Let us pick any points $x_n\in A_n$ such that $(x_n,A_n)\in F$.
Suppose that every $p\in A$ has an open neighborhood $U_p$ containing only finitely many of the $x_i$.
Since $A$ is compact, there exist finitely many points $p_1,\dots,p_r\in A$ such that 
$A\subset U\bydef U_{p_1}\cup\dots\cup U_{p_r}$.
Then $U$ contains only finitely many of the $x_i$.
On the other hand, since $A$ is compact and $U$ is open, there exists a $q\in\N$ such that $A_n\subset U$ 
for all $n\ge q$.
Hence $U$ contains $x_n$ for all $n\ge q$, which is a contradiction.
Thus our assumption was wrong, and some $x\in A$ is a cluster point of the sequence $x_i$.
Then $x$ is the limit of some subsequence $x_{n_i}$.
Hence $(x,A)$ is the limit of the sequence $(x_{n_i},A_{n_i})$. 
Since each $(x_i,A_i)\in F$ and $F$ is closed, we get that $(x,A)\in F$.
Hence $A\in\pi(F)$, which is a contradiction.

Thus $\pi$ is closed.
The second assertion follows from Lemma \ref{closed Leray} since $\pi^{-1}(A)$ is homeomorphic to $A$ for
each $A\in K(X)$.
\end{proof}

\begin{proof}[(b)]
Let us write $\h^n(\pi)$ as $\F\:E^n(X)\to K(X)$.
Let $A$ and $B$ be compact subsets of $X$ with $A\subset B$.
Let $\alpha\in\F_A=H^n(A)$ and $\beta\in\F_B=H^n(B)$ be such that $\F^B_A(\beta)=\alpha$,
or in other words, $\beta|_A=\alpha$.
Thus $B\le A$ in $K(X)$ and $\beta\le\alpha$ in $E^n(X)$.
By Spanier's tautness theorem (see \cite{Sp}*{Theorem 6.6.3}), $H^n(B)\simeq\lim_V H^n(V)$ over all open 
neighborhoods $V$ of $B$.
Hence there exists an open neighborhood $V$ of $B$ and a $\gamma\in H^n(V)$ such that $\gamma|_B=\beta$.
Since the Hausdorff metric induces the Vietoris topology on $K(X)$, the subset $W\bydef \{C\in K(X)\mid C\subset V\}$ 
of $K(X)$ is open in $K(X)$.
Actually, $W$ can be identified with $K(V)$, and we have $A,B\in W$.
Also, $\pi^{-1}(W)$ is an open subset of $E(X)$ that can be identified with $E(V)$.
Let $p$ be the composition $E(V)\subset V\x K(V)\to V$ of the inclusion and the projection, and let
$\sigma=p^*\gamma\in H^n\big(E(V)\big)$.
For each $C\in K(V)$, the composition $C=\pi^{-1}(C)\subset E(V)\xr{p} V$ coincides with
the inclusion map $C\to V$.
Hence $\sigma|_{\pi^{-1}(C)}=\gamma|_C$.
On the other hand, by Spanier's tautness theorem $\sigma|_{\pi^{-1}(C)}$ coincides with the image of 
$\sigma\in H^n\big(\pi^{-1}(W)\big)$ in $\F_C=\colim_U H^n\big(\pi^{-1}(U)\big)$, where $U$ runs over all open 
neighborhoods of $C$ in $K(X)$.
Since $\F$ is the sheafafication of the presheaf $U\mapsto H^n\big(\pi^{-1}(U)\big)$ and 
$\sigma\in H^n\big(\pi^{-1}(W)\big)$, there is a section $s\:W\to E^n(X)$ of $\F$ over $W$ given by 
$C\mapsto\gamma|_C\in\F_C$.

Let $s_\le$ be the restriction of $s\x s\:W\x W\to E^n(X)\x E^n(X)$ to $\le_W$.
Then $s_\le(\le_W)$ consists of all tuples $(B',A',\gamma|_{B'},\gamma|_{A'})$ such that $A'\subset B'$,
and in particular it contains $(B,A,\beta,\alpha)$.
But such a tuple $(B',A',\gamma|_{B'},\gamma|_{A'})$ satisfies $\gamma|_{A'}=\big(\gamma|_{B'}\big)|_{A'}$ and 
hence belongs to $\le_{E^n(X)}$. 
Thus $s_\le(\le_W)\subset \le_{E^n(X)}$.
On the other hand, $s_\le(\le_W)$ is open in $(\F\x\F)^{-1}(\le_{K(X)})$ since $s_\le$ is a section of 
the sheaf $\F\x\F|_{(\F\x\F)^{-1}(\le_{K(X)})}$ over $\le_W$.
Since $(B,A,\beta,\alpha)$ could be an arbitrary point of $\le_{E^n(X)}$, we conclude that $\le_{E^n(X)}$ is open in 
$(\F\x\F)^{-1}(\le_{K(X)})$.
\end{proof}

\subsection{Continuous derived limits}\label{holim-definition}

If $X$ is a topological poset, let $[X]$ denote the topological space with the same underlying set as $X$ and 
with $U$ open in $[X]$ if and only if it is open both in $X$ and in the Alexandrov topology corresponding to 
the order on $X$.
If $f\:X\to Y$ is a continuous monotone map between topological posets, it is continuous also as a map
$[f]\:[X]\to [Y]$.

Let $P$ be a pospace, and let $[n]$ denote the $n$-element poset $\{1,\dots,n\}$ with 
the usual (total) order.
A chain of length $n$ (without repeats) $p_1<\dots<p_n$ in $P$ may be regarded as 
an injective monotone map $[n]\to P$.
All such chains form a subspace $P^{[n]}$ of the product $P^n=P\x\dots\x P$ of spaces.
Let $\rho_n\:P^{[n]}\subset P^n\to P$ be the projection onto the last factor.
Let $P'$ be the topological poset $\bigsqcup_{n=1}^\infty P^{[n]}$ of all finite chains in $P$, ordered by inclusion,
with the topology of disjoint union.
The map $\rho\:P'\to P$ defined by $\rho(p_1<\dots<p_n)=p_n$ is continuous (since each $\rho_n$ is continuous) 
and, clearly, monotone.

Let $P^\Delta=[P']$.
Thus $U$ is open in $P^\Delta$ if and only if $U$ meets each $P^{[n]}$ in an open set, and
$c\in U$ implies $d\in U$ whenever $c$ is a subchain of $d$.
Let $\kappa\:|P|\to P^\Delta$ be defined by sending the interior of every simplex to that same simplex regarded as 
a point of $P'$. 

\begin{lemma} \label{combinatorialization}
If the topology of $P$ is Hausdorff, then $\kappa$ is continuous.
\end{lemma}

\begin{proof} Let $U$ be an open subset of $P^\Delta$ and let $x\in U$.
Then $x=(x_1<\dots<x_n)\in P^{[n]}\subset P^n$ for some $n$.
By the definition of product topology there exist neighborhoods $U_1,\dots,U_n$ of $x_1,\dots,x_n$ such that 
$(U_1\x\dots\x U_n)\cap P^{[n]}\subset U$.
Since $P$ is Hausdorff, we may assume that $U_1,\dots,U_n$ are pairwise disjoint.
Let $V_x$ be the set of all nonempty finite chains in $P$ that contain a subchain $y_1<\dots<y_n$ such that
each $y_i\in U_i$.
Since $(U_1\x\dots\x U_n)\cap P^{[n]}\subset U$ and $U$ is open in $P^\Delta$, we have $V_x\subset U$.
On the other hand, by \cite{M2}*{Lemma \ref{metr:open-in-HM}} (using that $P$ is $T_1$), $\kappa^{-1}(V_x)$ 
is open in $|P|$.
Hence also $\kappa^{-1}(U)=\bigcup_{x\in U}\kappa^{-1}(V_x)$ is open in $|P|$.
\end{proof}

Let $\F\:E\to P$ be a sheaf-diagram.
Let us consider the sheaf-diagram $\rho^*\F\:\rho^*E\to P'$.
Thus $\rho^*E$ consists of all tuples $(p_1<\dots<p_n;\,g)$, where $(p_1<\dots<p_n)\in P'$ and $g\in\F_{p_n}$, 
and has the topology of pullback of continuous maps and the order of pullback of monotone maps (namely, of the maps
$\F\:E\to P$ and $\rho\:P'\to P$).
Namely, $\rho^*E$ is homeomorphic to $\bigsqcup_{n=1}^\infty\rho_n^*E$ and is ordered by 
$(p_1<\dots<p_n;\,g)\le (q_1<\dots<q_m;\,h)$ if and only if $(p_1<\dots<p_n)$ is a subchain of $(q_1<\dots<q_m)$ 
and $\F^{p_n}_{q_m}(g)=h$.

Let $E^\Delta_\F=[\rho^*E]$ and let $\F^\Delta=[\rho^*\F]\:E^\Delta_\F\to P^\Delta$.
Since $\rho^*\F$ is monotone and continuous, $\F^\Delta$ is continuous.

Also $\F^\Delta$ is an open map.
Indeed, since $\rho^*\F$ is a sheaf, it is an open map.
Also $\rho^*\F$ is open with respect to the Alexandrov topologies on $\rho^*E$ and $P'$.
Indeed, if $U$ is open in the Alexandrov topology on $\rho^*E$, and some $(p_1<\dots<p_n)\in\rho^*\F(U)$
is a subchain of $(q_1<\dots<q_m)$, then $(q_1<\dots<q_n)\in\rho^*\F(U)$ since any $(p_1<\dots<p_n;\,g)\in U$ 
gives rise to $\big(q_1<\dots<q_m;\,\F^{p_n}_{q_m}(g)\big)\in U$.

\begin{lemma} \label{holim-sheaf}
If $\F$ is a posheaf, then $\F^\Delta$ is a sheaf.
\end{lemma}

\begin{proof}
For a subset $S$ of a poset $Q$ let $\cel S\cer$ denote the smallest subset of $Q$ that contains $S$ and is open
in the Alexandroff topology; in other words, $\cel S\cer$ consists of all $q\in Q$ such that $q\ge p$ for some
$p\in Q$.  

If $U$ is an open subset of $P^{[n]}$, then $\cel U\cer$ is open in $P^\Delta$.
Indeed, for each $k\ge 0$, $\cel U\cer\cap P^{[n+k]}$ is the union of the preimages of $U$ under the $\binom{n+k}k$ 
projections $P^{[n+k]}\to P^{[n]}$, which are continuous, being the restrictions of the projections $P^{n+k}\to P^n$.

Next, if $V$ is an open subset of $\rho_n^*E$, then $\cel V\cer$ is open in $E^\Delta_\F$.
Indeed, it suffices to show that $\cel V\cer\cap\rho_{n+k}^*E$ is open in $\rho_{n+k}^*E$ for each $k\ge 0$.
Let $\pi$ be one of the $\binom{n+k}k$ projections $P^{[n+k]}\to P^{[n]}$.
Let $f_\pi\:\pi^*\rho_n^*E\to\rho_{n+k}^*E$ be given by $f_\pi(c,g)=\big(c,\F^{\rho_n\pi(c)}_{\rho_{n+k}(c)}(g)\big)$.
Then $\pi^*\rho_n^*\F=(\rho_{n+k}^*\F)\circ f$, and since $\F$ is a posheaf, it follows from
Lemma \ref{posheaf}(b) that $f_\pi$ is continuous.
Thus $f_\pi$ is a homomorphism of sheaves $\pi^*\rho_n^*\F\to\rho_{n+k}^*\F$.
It is easy to see that every homomorphism of sheaves is itself a sheaf of sets; and that every sheaf of sets 
is an open map.
In particular, $f_\pi$ is an open map for each projection $\pi$.
Clearly, $\cel V\cer\cap\rho_{n+k}^*E$ is the union of the sets $f_\pi(V_\pi)$, where $V_\pi$ is the preimage
of $V$ under the natural map $\pi^*\rho_n^*E\to\rho_n^*E$.
Since the latter is continuous, each $V_\pi$ is open.
Thus $\cel V\cer\cap\rho_{n+k}^*E$ is open.

Since $\F$ is a sheaf, there exists an open neighborhood $V_{(p,g)}$ of each $(p,g)\in E$ such that $\F$ restricts to
a homeomorphism between $V_{(p,g)}$ and an open neighborhood $U_p$ of $p$ in $P$.
Since $\F$ is a posheaf, by Lemma \ref{posheaf}(a) for every $(p',g')\in V_{(p,g)}$ and $(p'',g'')\in V_{(p,g)}$ 
such that $p'\le p''$, we have $\F^{p'}_{p''}(g')=g''$.

Given an $(p_1<\dots<p_n;\,g)\in\pi^*E$, let $U$ and $V$ be the preimages of $U_{p_n}$ and $V_{(p_n,g)}$ under
$\rho_n\:P^{[n]}\to P$ and the natural map $\rho_n^*E\to E$.
Since $\rho_n^*E$ is clopen in $\rho^*E$, $V$ is an open neighborhood of $(p_1<\dots<p_n;\,g)$ in $\pi^*E$
and clearly $\rho^*\F$ restricts to a homeomorphism between $V$ and $U$.
By construction, for every $(p_1'<\dots<p_n';\,g')\in V$ and $(p_1''<\dots<p_n'';\,g'')\in V$ such that
$p_n'\le p_n''$, we have $\F^{p_n'}_{p_n''}(g')=g''$.

Now $\cel U\cer$ consists of all chains $(q_1<\dots<q_m)$ that have a subchain $(p_1'<\dots<p_n')$ in $U$;
and $\cel V\cer$ consists of all tuples $(q_1<\dots<q_m;\,h)$ such that $(q_1<\dots<q_m)$ has a subchain 
$(p_1'<\dots<p_n')$ in $U$, and $h=\F^{p_n'}_{q_m}(g')$, where $g'\in\F_{p_n'}$ is the unique element 
such that $(p_1'<\dots<p_n',\,g')\in V$.
Thus $h$ is uniquely determined by $(p_1'<\dots<p_n')$.
Given another $(p_1''<\dots<p_n'',\,g'')\in V$ such that $(p_1''<\dots<p_n'')$ is a subchain of $(q_1<\dots<q_m)$,
we have either $p_n''\le p_n'$ or $p_n'\le p_n''$, since $(q_1<\dots<q_n)$ is a chain.
By symmetry it suffices to consider the case $p_n'\le p_n''$.
Then $\F^{p_n'}_{p_n''}(g')=g''$, and consequently $\F^{p_n'}_{q_m}(g')=\F^{p_n''}_{q_m}(g'')$.
Thus $h$ does not depend on the choice of the subchain $(p_1'<\dots<p_n')$ and so is uniquely determined by 
the chain $(q_1<\dots<q_m)$.
Hence $\F$ restricts to a bijection between $\cel V\cer$ and $\cel U\cer$. 

Since $\F$ is an open continuous map and $\cel V\cer$ is open, $\F$ restricts to an open continuous map
between $\cel V\cer$ and $\cel U\cer$.
Since this restriction is also a bijection, it must be a homeomorphism.
\end{proof}

We define $\Holim\F$ to be the pullback $\kappa^*\F^\Delta\:\kappa^*E^\Delta_\F\to |P|$.
In other words,
\[\Holim\F=\kappa^*[\rho^*\F^*]\:\kappa^*[\rho^*E]\to|P|.\]
Lemmas \ref{combinatorialization} and \ref{holim-sheaf} imply that $\Holim\F$ is a sheaf.

Given a posheaf $\F$ of abelian groups over a topological poset $P$, we define $\Lim^n\F$, also denoted
$\underset{p\in P}\Lim^n\F_p$, to be $H^n(|P|;\,\Holim\F)$ for each $n=0,1,\dots$.
As discussed in \S\ref{discrete}, when $P$ has discrete topology, these are the usual derived limits.

\subsection{Leray sheaves}

Let us call a diagram of spaces $\pi\:E\to P$ an {\it order-isometry} if there exist metrics 
on $E$ and $P$, compatible with the topologies, such that $d\big(\pi(x),\pi(x')\big)=d(x,x')$ whenever $x\le x'$.

\begin{lemma} \label{order-isometry}
Let $\pi\:E\to P$ be a diagram of spaces, with metrizable $E$ and $P$.
If $\pi$ is a closed map and an order-isometry, then $\Hocolim\pi\:|E|\to|P|$ is a closed map.
\end{lemma}

\begin{proof} Suppose that $F\subset |E|$ is closed, but $(\Hocolim\pi)(F)$ is not.
Then there exists a sequence of points $y_1,y_2,\dots\in(\Hocolim\pi)(F)$ converging to a point 
$y\notin(\Hocolim\pi)(F)$.
We have $y=\sum_{i=1}^n\lambda_i y_i\in |P|$, where $y_1<\dots<y_n$, each $y_i\in P$, each $\lambda_i\ge 0$ and 
$\sum_{i=1}^n\lambda_i=1$.
Then for each $k$ we have $y_k=\sum_{i=1}^n\sum_{j=1}^{m_{ki}}\lambda_{kij} y_{kij}\in |P|$, where 
$y_{k,i,1}<\dots<y_{k,i,m_{ki}}\le y_{k,i+1,1}$ for each $k$ and $i$, each $y_{kij}\in P$, and 
$\sum_{j=1}^{m_{ki}}\lambda_{kij}=\lambda_i$.
Since $y_{k,i,1}<\dots<y_{k,i,m_{ki}}$, the formal sum $\frac1{\lambda_i}\sum_{j=1}^{m_{ki}}\lambda_{kij} y_{kij}$ 
denotes a point $y_{ki}\in |P|$.
Hence we can write formally $y_k=\sum_{i=1}^n\lambda_i y_{ki}$; we will further take this formal equation to encode 
the more useful observation that the step function $\phi_{y_k}$ is the ``stacked linear combination'' of 
the step functions $\phi_{y_{ki}}$ in the sense that 
$\phi_{y_k}(\lambda_1+\dots+\lambda_{i-1}+t\lambda_i)=\phi_{y_{ki}}(t)$ for $t\in [0,1)$.
Therefore $L_1(\phi_y,\phi_{y_k})=\sum_{i=1}^n\frac1{\lambda_i} L_1(\phi_{y_i},\phi_{y_{ki}})$, and consequently 
for each $i=1,\dots,n$ the sequence $y_{ki}$ converges to $y_i$.
Here each $y_i\in P$, but the $y_{ki}$ need not lie in $P$; however, each $y_{kij}\in P$, and for each $k$ and $i$
we may choose one of these points $y'_{ki}\bydef y_{k,i,j_{ki}}\in P$ so that the sequence $y'_{ki}$ also converges 
to $y_i$ for each $i$ (see \cite{M2}*{Lemma \ref{metr:open-in-HM}}).
Then, in particular, $d(y_{ki},y'_{ki})\to 0$ as $k\to\infty$.

Let us pick some points $x_k\in(\Hocolim\pi)^{-1}(y_k)\cap F$.
Then for each $k$ we have $x_k=\sum_{i=1}^n\sum_{j=1}^{m_{ki}}\lambda_{kij} y_{kij}\in |E|$, where 
$x_{k,i,1}<\dots<x_{k,i,m_{ki}}\le x_{k,i+1,1}$ for each $k$ and $i$ and each 
$x_{kij}\in(\Hocolim\pi)^{-1}(y_{kij})$.
(Of course, the $x_{ijk}$ do not necessarily lie in $F$.)
Again the formal sum $\frac1{\lambda_i}\sum_{j=1}^{m_{ki}}\lambda_{kij} x_{kij}$ denotes a point $x_{ki}\in |E|$,
and $x_k=\sum_{i=1}^n\lambda_i x_{ki}$ (which encodes a relation between the step functions).
Let $x'_{ki}=x_{k,i,j_{ki}}\in E$; thus $y'_{ki}=\pi(x'_{ki})$.
Since $\pi$ is a closed map and $\{y'_{k,1}\mid k\in\N\}$ is not closed in $P$,
$\{x'_{k,1}\mid k\in\N\}$ is not closed in $E$.
Hence some subsequence $x'_{k_l,1}$ converges to a point $x_1\in E$.
Since $\pi$ is a closed map and $\{y'_{k_l,1}\mid l\in\N\}$ is not closed in $P$,
$\{x'_{k_l,2}\mid l\in\N\}$ is not closed in $E$.
Hence some subsequence $x'_{k_{l_m},2}$ converges to a point $x_2\in E$.
By arguing in the same fashion, we will construct a sequence of numbers $\kappa_1,\kappa_2,\dots\in\N$ such that
for each $i=1,\dots,n$ the sequence $x'_{\kappa_l,i}$ converges to a point $x_i\in E$.
Since both $x_{ki}$ and $x'_{ki}$ belong to the simplex of $|E|$ spanned by the chain 
$x_{k,i,1}\le\dots\le x_{k,i,m_{ki}}$ and $\pi$ is an order-isometry, $d(x_{ki},x'_{ki})=d(y_{ki},y'_{ki})$.  
Hence $d(x_{ki},x'_{ki})\to 0$ as $k\to\infty$.
Therefore the sequence $x_{\kappa_l i}$ also converges to $x_i$.
Since each $x_k=\sum_{i=1}^n\lambda_i x_{ki}$, the sequence $x_{\kappa_l}$ converges to 
$x\bydef \sum_{i=1}^n\lambda_i x_i$.
Since each $x_k\in F$ and $F$ is closed, $x\in F$.
On the other hand, since $\Hocolim\pi$ is continuous, $(\Hocolim\pi)(x)=y$.
Thus $y\in(\Hocolim\pi)(F)$, which is a contradiction.
\end{proof}

\begin{theorem} \label{Leray}
Let $X$ be a metrizable space and let $E(X)$ be the subspace of $X\x K(X)$ consisting of all pairs 
$(x,K)$ such that $x\in K$.
Let $\pi$ be the composition of the inclusion $E(X)\subset X\x K(X)$ and the projection $X\x K(X)\to K(X)$.

(a) $\Hocolim\pi\:|E(X)|\to|K(X)|$ is a closed map.

(b) $\h^n(\Hocolim\pi)\simeq\Holim\h^n(\pi)$ (isomorphism of sheaves).
\end{theorem}

Here $K(X)$ again must be ordered by {\it reverse} inclusion in (b) in order for $\h^n(\pi)$ to be a posheaf.
In fact, it will be convenient to work with the usual order by inclusion on $K(X)$; in this notation,
the monotone map $\rho\:\big(K(X)^*\big)'\to K(X)^*$, $(A_n<\dots<A_1)\mapsto A_1$ (where $<$ means $\supsetneqq$), 
corresponds to the antitone map $\bar\rho\:K(X)'\to K(X)$ given by $(A_1<\dots<A_n)\mapsto A_1$ 
(where $<$ means $\subsetneqq$).

\begin{proof}[Proof. (a)] 
Let us choose some metrics on $X$ and $K(X)$ and consider, for example, the $l_\infty$ product metric
on $E(X)\subset X\x K(X)$.
Then the projection $\pi\:E(X)\to K(X)$ is clearly an order-isometry.
Now the assertion follows from Lemma \ref{order-isometry}.
\end{proof}

\begin{proof}[(b)] Let $\h=\h^n(\Hocolim\pi)$ and $\h'=\Holim\h^n(\pi)$.
Given an $A\in |K(X)|$, we have $A=\sum_{i=1}^n\lambda_i A_i$, where $A_1\subset\dots\subset A_n$ are 
pairwise distinct compact subsets of $X$, each $\lambda_i>0$ and $\sum_{i=1}^n\lambda_i=1$.
Over the simplex $(A_1<\dots<A_n)$ of $K(X)$, $\Hocolim\pi$ is the projection of the iterated mapping
cylinder $\cyl(A_1\subset\dots\subset A_n)$ onto that simplex; in particular,
$(\Hocolim\pi)^{-1}(A)\cong\pi^{-1}(A_1)\cong A_1$.
Hence by (a) and Lemma \ref{closed Leray}, $\h_A\simeq H^n(A_1)$.
On the other hand, using the notation of \S\ref{holim-definition} with $P=K(X)^*$, we have 
$\kappa(A)=(A_n<\dots<A_1)$ and $\rho(A_n<\dots<A_1)=\bar\rho(A_1<\dots<A_n)=A_1$.
Since $\pi^{-1}(A_1)\cong A_1$, we have $\h'_A\simeq H^n(A_1)$.
Thus $\h$ and $\h'$ have isomorphic stalks.
So their \'etale spaces can be identified as sets, also with identical group structures in each stalk, and 
it remains to verify that they have the same topology.
For that it suffices to show that for every $A\in|K(X)|$ and every $\alpha\in\h_A=\h'_A$, there exists an open 
neighborhood $U$ of $A$ in $|K(X)|$ and sections $s$, $s'$ of $\h$, $\h'$ over $U$ such that $s(A)=s'(A)=\alpha$
and $s(C)=s'(C)\in\h_C=\h'_C$ for each $C\in U$.

Let $A=\sum_{i=1}^n\lambda_i A_i$ be as above.
By Spanier's tautness theorem (see \cite{Sp}*{Theorem 6.6.3}), $H^n(A_1)\simeq\lim_V H^n(V)$ over all open 
neighborhoods $V$ of $A_1$ in $X$.
Hence there exists an open neighborhood $V$ of $A_1$ and a $\gamma\in H^n(V)$ such that $\gamma|_{A_1}=\alpha$.
Since the Hausdorff metric induces the Vietoris topology on $K(X)$, the subset $W\bydef \{B\in K(X)\mid B\subset V\}$ 
of $K(X)$ is open in $K(X)$.
Actually, $W$ can be identified with $K(V)$, and we have $A\in W$.
Since $\bar\rho$ is continuous, $\bar\rho^{-1}(W)=\{(C_1<\dots<C_n)\in K(X)'\mid C_1\subset V\}$ is open in
$K(X)'$.
Since a chain $C_1\subset\dots\subset C_n$ of $K(X)$ satisfies $C_i\subset V$ for some $i$ if and only if
it satisfies $C_1\subset V$, in fact $\bar\rho^{-1}(W)$ is also Alexandrov open in $K(X)'$.
Hence $[\bar\rho^{-1}(W)]$ is an open subset of $[K(X)]$.
The open subset $U\bydef \kappa^{-1}[\bar\rho^{-1}(W)]$ of $|K(X)|$ consists of all $C=\sum_{i=1}^n\lambda_iC_i$ 
such that $(C_1<\dots<C_n)\in K(X)'$, $C_1\subset V$, each $\lambda_i>0$ and 
$\sum_{i=1}^n\lambda_i=1$.

By the proof of Theorem \ref{hyperspace-posheaf}(b), there is a section $s_\F$ of the Leray sheaf $\F\bydef \h^n(\pi)$ 
over $W$ given by $B\mapsto\gamma|_B\in\F_B$.
Then $s_\F\bar\rho$ is a section of $\bar\rho^*\F$ over $\bar\rho^{-1}(W)$ given by
$(C_1<\dots<C_n)\mapsto\gamma|_{C_1}\in\F_{C_1}=(\bar\rho^*\F)_{(C_1<\dots<C_n)}$.
In the alternative language, $s_\F\rho$ is a section of $\rho^*\F^*$ over $\rho^{-1}(W^*)$ given by
$(C_n<\dots<C_1)\mapsto\gamma|_{C_1}\in(\F^*)_{C_1}=(\rho^*\F^*)_{(C_n<\dots<C_1)}$.
Since $s$ and $\rho$ are monotone and continuous (see the proof of Theorem \ref{hyperspace-posheaf}(b)
concerning the monotonicity of $s$), $[s\rho]$ is continuous and hence is a section of $[\rho^*\F^*]$
over $[\rho^{-1}(W^*)]$.
But we have $[s\rho]=[s\bar\rho]$, $[\rho^*\F^*]=[\bar\rho^*\F]$ and $[\rho^{-1}(W^*)]=[\bar\rho^{-1}(W)]$.
Therefore $s'\bydef [s\bar\rho]\kappa$ is a section of $\h'=\kappa^*[\bar\rho^*\F]$ over 
$U=\kappa^{-1}[\bar\rho^{-1}(W)]$.
Clearly, $s'$ is given by $C\mapsto\gamma|_{C_1}\in\F_C$, where $C$ is as above.
In particular, $s'(A)=\gamma|_{A_1}=\alpha$, as desired.

The open subset $(\Hocolim\pi)^{-1}(U)$ of $|E(X)|\subset X\x |K(X)|$ consists of all pairs $(x,C)$
where $C=\sum_{i=1}^n\lambda_iC_i\in U$ and $x\in C_1$.
Hence $(\Hocolim\pi)^{-1}(U)$ lies in $V\x |K(X)|$.
Let $p$ be the composition $(\Hocolim\pi)^{-1}(U)\subset V\x |K(X)|\to V$ of the inclusion and the projection, 
and let $\sigma=p^*\gamma\in H^n\big((\Hocolim\pi)^{-1}(U)\big)$.
For each $C=\sum_{i=1}^n\lambda_iC_i\in U$, the composition $C_1=|\pi|^{-1}(C)\subset (\Hocolim\pi)^{-1}(U)\xr{p} V$ 
coincides with the inclusion map $C_1\to V$.
Hence $\sigma|_{(\Hocolim\pi)^{-1}(C)}=\gamma|_{C_1}$.
On the other hand, by Spanier's tautness theorem $\sigma|_{(\Hocolim\pi)^{-1}(C)}$ coincides with the image of 
$\sigma\in H^n\big((\Hocolim\pi)^{-1}(U)\big)$ in $\h_C=\colim_O H^n\big((\Hocolim\pi)^{-1}(O)\big)$, where $O$ 
runs over all open neighborhoods of $C$ in $|K(X)|$.
Since $\h$ is the sheafafication of $O\mapsto H^n\big((\Hocolim\pi)^{-1}(O)\big)$ and 
$\sigma\in H^n\big((\Hocolim\pi)^{-1}(U)\big)$, there is a section $s$ of $\h$ over $U$ given by 
$C\mapsto\gamma|_{C_1}\in\F_C$.
In particular, $s(A)=\gamma|_{A_1}=\alpha$, as desired.
\end{proof}

\subsection{Spectral sequences}

\begin{theorem} \label{lerayss}
Let $X$ be a metrizable space and let $K(X)$ be the pospace of its nonempty compact subsets,
topologized by the Hausdorff metric and ordered by reverse inclusion.
Then there is a spectral sequence of the form
\[\underset{K_\alpha\in K(X)}{\Lim^p}H^q(K_\alpha)\Rightarrow H^{p+q}(X).\]
\end{theorem}

In more detail, the theorem asserts that for each $q\ge 0$ there is a posheaf $\F^q$ over $K(X)$ such that
the stalk $\F^q_{K_\alpha}\simeq H^q(K_\alpha)$ for each $K_\alpha\in K(X)$, and there is a spectral sequence
of the form $\Lim^p\F^q=E_2^{pq}\ \Rightarrow\ H^{p+q}(X)$.

\begin{proof} Let $E(X)$ be the subspace of $X\x K(X)$ consisting of all pairs $(x,K)$ such that $x\in K$.
Let $\pi$ be the composition of the inclusion $E(X)\subset X\x K(X)$ and the projection $X\x K(X)\to K(X)$.

Then we get a continuous map $\Hocolim\pi\:|E(X)|\to|K(X)|$ (see \S\ref{posheaves}).
The Leray spectral sequence of this map runs 
$H^p\big(|K(X)|;\,\h^q(\Hocolim\pi)\big)\Rightarrow H^{p+q}\big(|E(X)|\big)$ \cite{Br}.

By Theorem \ref{atoms}, $|E(X)|$ is homotopy equivalent to $X$.

By Theorem \ref{Leray}(b), the Leray sheaf $\h^q(\Hocolim\pi)\simeq\Holim\h^q(\pi)$.
By definition, $\Lim^p\h^q(\pi)=H^p\big(|K(X)|;\,\Holim\h^q(\pi)\big)$.

By Theorem \ref{hyperspace-posheaf}(a), the stalk $\h^q(\pi)_{K_\alpha}\simeq H^n(K_\alpha)$ for each compact
$K_\alpha\subset X$.
\end{proof}

Every finite-dimensional separable metrizable space embeds in some sphere $S^n$.
For subsets of $S^n$ it is easy to rewrite the previous spectral sequence in terms of homology:

\begin{theorem} \label{lerayss2}
Let $X$ be a subset of $S^n$ and let $U(X)$ be the pospace of its open neighborhoods
$\ne S^n$, ordered by inclusion and topologized by the Hausdorff metric.
Then there is a spectral sequence of the form
\[\underset{U_\alpha\in U(X)}{\Lim^p}H_q(U_\alpha)\Rightarrow H_{q-p}(X).\]
\end{theorem}

In more detail, the theorem asserts that for each $q\ge 0$ there is a posheaf $\F_q$ over $K(X)$ such that
the stalk $(\F_q)_{P_\beta}\simeq H_q(P_\beta)$ for each $P_\beta\in U(X)$, and there is a second quadrant
homology spectral sequence of the form $\Lim^p\F_q=E^2_{-p,q}\ \Rightarrow\ H_{q-p}(X)$ (where $p,q\ge 0$).

\begin{proof}
Clearly, $U(X)$ is homeomorphic to $K(S^n\but X)$.
By the Sitnikov duality (see \cite{M00}*{Theorem \ref{book:alex duality}}) we have 
$H_i(X)\simeq H^{n-i-1}(S^n\but X)$ and $H_i(U)\simeq H^{n-i-1}(S^n\but U)$ 
for every open neighborhood of $X$.
So the assertion follows from Theorem \ref{lerayss}.
\end{proof}

\section{Discussion}

The results of the present paper suggest that it may be worthwhile to develop an entire theory of derived limits 
over posets.
It must be admitted that the present state of this theory is absolutely unsatisfactory, if not to say 
downright pathetic.

The following obvious problems still wait to be addressed:
\begin{enumerate}
\item Compute $\Lim^p$ for a few basic examples.

\item How does $\Lim^p$ behave with respect to cofinal subsets?

\item When does $\Lim^0$ coincide with the usual (discretely indexed) inverse limit $\lim$?

It is not hard to show that they do coincide in our model setting (that is, for the Leray sheaves 
$\h^q(\pi)$ of the map $\pi\:E(X)\to K(X)$) as long as $X$ is locally compact.
Which is hardly surprising, but even this is not entirely obvious in the absence of answers to 
the previous question.

On the other hand, there seems to be no reason to expect that $\Lim^0=\lim$ in full generality.
Or if they do always coincide, that would actually be pretty bad!
If $G_\alpha$ is the Marde\v si\'c--Prasolov inverse system (whose $\lim^1$ cannot be computed in ZFC)
and $S_\alpha=\Hom\big(\Hom(G_\alpha,S^1),S^1\big)$, we get a short exact sequence of inverse systems
$0\to G_\alpha\to S_\alpha\to Q_\alpha\to 0$, where $\lim Q_\alpha$ cannot be computed in ZFC.

\item When does $\Lim^p=\lim^p$?

\item Obtain vanishing results for $\Lim^p$.

\item What can be said of $\Lim^p$ over zero-dimensional pospaces?

\item Understand $\Lim^p$ as derived functors. 
Deduce a long exact sequence and a formula using an explicit resolution.

\item Compute $\Lim^p$ for the telescopic chain complexes of \cite{M-II}.

\item Apply $\Lim^p$ to prove something interesting.
In particular, can one establish a version of \cite{M-III}*{Theorem \ref{lim:lim2theorem}}
not involving any assumptions independent from ZFC?
\end{enumerate}

\subsection*{Acknowledgements}

I would like to thank A. Dranishnikov, M. Jibladze and D.~ Saveliev for stimulating feedback.

\subsection*{Disclaimer}

I oppose all wars, including those wars that are initiated by governments at the time when 
they directly or indirectly support my research. The latter type of wars include all wars 
waged by the Russian state in the last 25 years (in Chechnya, Georgia, Syria and Ukraine) 
as well as the USA-led invasions of Afghanistan and Iraq.

\end{document}